\documentclass[11pt]{article}

\usepackage[a4paper,margin=.7in]{geometry}
\usepackage[T1]{fontenc}
\usepackage[utf8]{inputenc}
\usepackage{amsmath,amssymb,amsthm,mathtools}
\usepackage{microtype}
\usepackage{enumitem}
\usepackage[hidelinks]{hyperref}

\newtheorem{theorem}{Theorem}[section]
\newtheorem{proposition}[theorem]{Proposition}
\newtheorem{lemma}[theorem]{Lemma}

\newtheorem{definition}[theorem]{Definition}
\newtheorem{assumption}[theorem]{Assumption}
\theoremstyle{remark}
\newtheorem{remark}[theorem]{Remark}

\newcommand{\R}{\mathbb{R}}
\newcommand{\eps}{\varepsilon}
\newcommand{\nupar}{\nu}
\newcommand{\dd}{\,\mathrm{d}}
\newcommand{\norm}[1]{\left\lVert #1\right\rVert}
\newcommand{\abs}[1]{\left\lvert #1\right\rvert}
\newcommand{\cK}{\mathcal{K}}
\newcommand{\cT}{\mathcal{R}}
\newcommand{\cE}{\mathcal{E}}
\newcommand{\cN}{\mathcal{N}}

\title{Regularized Subjective-Surface Flow with Monotone Reaction: Global Classical Well-Posedness and Stability}

\author{Markjoe O. Uba\\
\small School of Mathematical and Statistical Sciences, Northern Illinois University, DeKalb, IL 60115, USA\\
\small \texttt{markjoeuba@gmail.com}}
\date{}

\begin{document}
\maketitle

\begin{abstract}
Touching and dividing cell nuclei may appear as connected structures in microscopy images, making it difficult to distinguish neighboring nuclei during segmentation. We introduce and analyze a regularized subjective-surface model designed for this setting. On a smooth bounded domain $\Omega\subset\mathbb{R}^n$ with homogeneous Dirichlet boundary conditions, the model is
\[
u_t
=
\nu\Delta u
+
\left(\varepsilon^2+|\nabla u|^2\right)^{1/2}
\operatorname{div}\!\left(
G(x)\frac{\nabla u}
{\left(\varepsilon^2+|\nabla u|^2\right)^{1/2}}
\right)
-
\mu\Lambda(x)H_\eta(u-q).
\]
Here, $\varepsilon>0$ and $\nu>0$ are fixed regularization parameters, $G$ is a strictly positive smooth edge coefficient, and the nonnegative interaction weight $\Lambda$ incorporates fixed information about neighboring nucleus candidates. The principal objective of this work is to establish an existence theory for the proposed model. For compatible $C^{2+\alpha}$ initial data taking values in $[0,1]$, we prove the existence and uniqueness of a global classical solution whose restriction to every finite time interval is Schauder-classical, together with preservation of the physical range, finite-time Schauder estimates, and $L^\infty$-nonexpansive dependence on the initial data. The main analytical step is a global spatial-gradient bound, obtained by combining gradient estimates near $\partial\Omega$ with interior gradient estimates. These results provide a mathematical foundation for applying the model to the analysis of touching and dividing nuclei in 3D and 3D+time microscopy image data.
\end{abstract}



\section{Introduction}

Touching and dividing cell nuclei motivate this work. When nuclei touch, their visible boundaries may merge or become weak in the contact region; during division, the evolving nuclear geometry must be represented coherently across space and, for time-dependent data, across successive frames. The modeling objective is to describe touching and dividing nuclei by a single segmentation function while incorporating fixed information about neighboring nucleus candidates.

Subjective-surface and level-set methods represent the segmentation boundary as an isosurface of a segmentation function. The function evolves in an artificial segmentation time until the selected isosurface captures the nuclear region of interest. The classical subjective-surface model and the 4D segmentation method of Uba, Mikula, and Park evolve one segmentation function; in the 4D model the edge detector combines the original image intensity with a locally thresholded image intensity \cite{SartiMalladiSethian2000,umknskp2020,UbaMikulaPark2023}.

We propose a regularized subjective-surface equation for touching and dividing nuclei. The image-dependent coefficient $G$ guides the evolving surface toward nuclear boundaries, while a graph-based preprocessing criterion constructs a fixed nonnegative interaction weight $\Lambda\in C^{1+\alpha}(\overline\Omega)$ from neighboring candidate nuclei. The interaction weight $\Lambda$ enters a monotone reaction term that penalizes occupation in regions associated with neighboring nuclei.

The principal mathematical objective of the paper is to prove that the new model possesses a unique global classical solution. The continuum model is one PDE for one segmentation function, and the neighboring-nucleus information is encoded before the evolution in $\Lambda(x)$.

For a 3D+time image, we write
\[
        x=(x_1,x_2,x_3,\theta),
\]
where $\theta$ denotes real video time in the data set, whereas $t$ denotes artificial segmentation time.  The unknown remains a single function $u(t,x)$.  The 4D image is therefore a coefficient function on a four-dimensional computational domain.

The equation studied below is a strictly parabolic regularization of the subjective-surface flow.  The Evans--Spruck factor, in the spirit of the level-set regularization in \cite{EvansSpruck1991}, removes the singularity at $\nabla u=0$, and the strictly positive linear diffusion term $\nupar\Delta u$ gives a uniform ellipticity constant independent of the size of $\nabla u$.  The main theorem establishes a fully classical uniformly parabolic theory for fixed $\nupar>0$ and $\eps>0$.

The continuum theorem is formulated on a smooth bounded domain. In computation, rectangular voxel and doxel grids are equipped with the selected numerical boundary closure. The model combines microscopy intensity, local thresholding, and neighboring-nucleus coefficient functions on the segmentation domain.


The main theorem gives the well-posedness foundation for the proposed nuclei model: global Schauder-classical existence and uniqueness, preservation of $[0,1]$, finite-time $C^1$ and H\"older estimates, and quantitative $L^\infty$ stability for fixed $\nupar>0$ and $\eps>0$. The finite-time gradient estimate is obtained by combining gradient estimates near $\partial\Omega$ with interior gradient estimates, while monotonicity of the neighboring-nucleus reaction profile yields the fixed-coefficient supremum-norm nonexpansive estimate.

\section{Construction of the prescribed coefficient functions}
\label{sec:construction}

Let $d\in\{2,3,4\}$ and let $\Omega\subset\R^d$ denote the continuum image domain used in the analysis.  For a single 3D image, $d=3$ and $x=(x_1,x_2,x_3)$.  For a 3D+time data set, the natural data region is
\[
        Q_{\mathrm{data}}=\Omega_x\times(0,\theta_F),\qquad x=(y,\theta),\qquad y\in\Omega_x\subset\R^3,
\]
where $\theta$ is the real video time coordinate.  The classical theorem below uses a smooth bounded four-dimensional domain representing this data region, still denoted by $\Omega$.  The artificial segmentation time is always denoted by $t$.

\subsection{Image-dependent edge coefficient}

Let $I^0$ be the given image intensity on the data region, transferred after smoothing to the continuum domain $\Omega$.  The local thresholding used in the 4D segmentation algorithm is a preprocessing step that produces one thresholded image on the same 4D data region before the final smoothed coefficient is formed.  More precisely, for each real video time $\theta\in(0,\theta_F)$ one selects spatial centers $c_m^\theta\in\Omega_x$ and spatial balls
\[
        B_x(c_m^\theta,r)\subset\Omega_x,\qquad m=1,\dots,N_\theta.
\]
For each such ball define
\[
        \alpha_m^\theta=\min_{y\in B_x(c_m^\theta,r)} I^0(y,\theta),\qquad
        \beta_m^\theta=\max_{y\in B_x(c_m^\theta,r)} I^0(y,\theta),
\]
and
\[
        TH_m^\theta=\lambda\alpha_m^\theta+(1-\lambda)\beta_m^\theta,
        \qquad 0\le \lambda\le1.
\]
Then, for $y\in B_x(c_m^\theta,r)$,
\[
        I^{\mathrm{TH}}(y,\theta)=
        \begin{cases}
        \beta_m^\theta, & I^0(y,\theta)\ge TH_m^\theta,\\
        \alpha_m^\theta, & I^0(y,\theta)< TH_m^\theta.
        \end{cases}
\]
If the selected balls overlap, the computational segmentation method fixes the assignment criterion. If a voxel lies outside all selected balls, the method either leaves it unchanged or assigns it a background value. The 3D construction is the same with the real video time variable omitted.

A standard generalized edge detector has the form
\begin{equation}\label{eq:edge-template}
        G^0(x)=g\!\left(\delta_0\,\abs{\nabla(G_\sigma*I^0)(x)}+
        \vartheta_0\,\abs{\nabla(G_\sigma*I^{\mathrm{TH}})(x)}\right),
\end{equation}
where $G_\sigma$ is a smoothing kernel, $g$ is positive and nonincreasing, and $\delta_0,\vartheta_0\in[0,1]$ determine the relative influence of the original and locally thresholded image information \cite{umknskp2020,UbaMikulaPark2023}.  Formula \eqref{eq:edge-template} constructs the coefficient used in the evolution equation.

The piecewise-constant thresholded image $I^{\mathrm{TH}}$ is smoothed and converted into continuous functions, with a positive lower bound imposed to obtain the regular coefficient functions used in the continuum model.
  In 3D+time applications, this also includes smoothing in the real-time coordinate $\theta$ whenever frame-by-frame preprocessing would otherwise produce temporal discontinuities.
  This is consistent with the numerical practice of presmoothing the image. The resulting single edge coefficient is denoted below by $G$.

\subsection{Prescribed interaction weight from candidate graph data}
\label{subsec:interaction-weight}

The prescribed interaction weight is built before the PDE is evolved. Let $c_0$ denote the center of the current candidate, and let $\{c_j,r_j\}$ denote nearby candidate centers and scales obtained from the detection or seeding stage. A typical candidate-neighborhood criterion is
\begin{equation}\label{eq:graph-criterion}
j\in\cN(c_0) \quad\Longleftrightarrow\quad
|c_j-c_0|\le \kappa(r_j+r_0),
\end{equation}
with $\kappa>0$. This graph-based criterion selects the neighboring candidates that influence the current segmentation through the prescribed coefficient function.

In a 3D+time data set, the distance used in \eqref{eq:graph-criterion} must first be defined by specifying how spatial and temporal separations are combined. This choice determines the neighboring candidates used to construct the coefficient function. One common choice is to build the candidate graph frame-by-frame using only spatial centers $c_j^\theta\in\Omega_x$. Another choice is to use a scaled spacetime metric, for instance $|(y_j-y_0,\lambda_\theta(\theta_j-\theta_0))|$, where the conversion factor $\lambda_\theta>0$ is prescribed by the preprocessing method. After smoothing, the selected graph-based criterion defines a fixed function $\Lambda$ with the regularity required in Assumption~\ref{ass:data}. Throughout the analytical theorem, the candidate-neighborhood set $\cN(c_0)$ is assumed to be finite. Equivalently, if the implementation initially produces a larger or image-selected candidate set, the preprocessing step first selects a fixed finite subset. The coefficient $\Lambda$ is then constructed from this selected set.

For each $j\in\cN(c_0)$, let $\omega_j\in C^{2+\alpha}(\overline\Omega)$ be a nonnegative spatial graph weight, for example a smoothed Gaussian mask centered near the overlap region between the current candidate and the neighbor $j$.  For each candidate $j$, let
$\widehat u_j\in C^{2+\alpha}(\overline\Omega)$
be a prescribed function describing the region associated with that candidate. It may be obtained from seeding, thresholding, or an earlier segmentation step. Both $\omega_j$ and $\widehat u_j$ are determined before the evolution begins and remain fixed throughout the analysis.
  Since $\cN(c_0)$ is finite, the following sum is a finite $C^{2+\alpha}$ coefficient.  Define the interaction-density function
\begin{equation}\label{eq:topology-density}
        \Lambda(x)=\sum_{j\in\cN(c_0)}\omega_j(x)\,H_\eta(\widehat u_j(x)-q),
\end{equation}
where the decision level $q$ and transition width $\eta$ are specified below.  Thus $\Lambda$ is large in regions where graph-neighboring candidates have substantial prescribed candidate support.  In 3D+time applications, framewise graph data and candidate neighborhoods must be smoothed in the real-time coordinate before they are regarded as a coefficient satisfying the analytical regularity hypothesis.  If $\omega_j,\widehat u_j\in C^{2+\alpha}(\overline\Omega)$ in the finite sum above, then the construction gives the stronger property $\Lambda\in C^{2+\alpha}(\overline\Omega)$, although the PDE theorem assumes $C^{1+\alpha}$ regularity of $\Lambda$.  The choice $\Lambda\equiv0$ or $\mu=0$ recovers the regularized subjective-surface flow without the graph-guided reaction term.

Choose a decision level $q\in(0,1)$ and a transition width
\[
        0<\eta<\min\{q,1-q\}.
\]
Let ${H}\in C^\infty(\R)$ satisfy
\[
        0\le H\le1,
        \qquad H'\ge0,
        \qquad H(s)=0\ \text{for }s\le-1,
        \qquad H(s)=1\ \text{for }s\ge1,
\]
and set
\begin{equation}\label{eq:profiles}
        H_\eta(s)=H(s/\eta),\qquad
        h_\eta(s)=H_\eta'(s)=\eta^{-1}H'(s/\eta),
        \qquad
        P_\eta(s)=\int_{-\infty}^{s}H_\eta(r)\,\dd r .
\end{equation}
Then $H_\eta$ is a smooth nondecreasing transition from $0$ to $1$, while
$h_\eta\ge0$ is supported in the transition band $|s|<\eta$. The modeling interpretation of this threshold, together with a practical discussion of the choice of $q$, is given in Appendix~\ref{app:model-interpretation}.  The monotonicity condition $H_\eta'\ge0$ makes the reaction term proper in the unknown and yields the comparison principle and the nonexpansive stability theorem.  The interaction energy is
\begin{equation}\label{eq:scalar-overlap-energy}
        \cE^\eta(u)=\int_\Omega \Lambda(x)P_\eta(u(x)-q)\,\dd x.
\end{equation}
Its $L^2$-gradient density for smooth functions is
\begin{equation}\label{eq:topology-reaction}
        \cT_\eta(x,u)=\Lambda(x)H_\eta(u-q).
\end{equation}
Indeed,
\begin{equation}\label{eq:scalar-variation}
        D\cE^\eta(u)[\phi]
        =\int_\Omega \cT_\eta(x,u(x))\phi(x)\,\dd x.
\end{equation}
The monotone reaction subflow $u_t=-\mu\cT_\eta(x,u)$ dissipates \eqref{eq:scalar-overlap-energy}. The term
$\cT_\eta$ is a smooth decision-level penalty: it vanishes below
$q-\eta$, increases smoothly for $u\in[q-\eta,q+\eta]$, and remains active
above $q+\eta$. Thus it penalizes occupation in regions selected by the prescribed interaction weight, both in saturated regions and near the final isosurface.

\section{The regularized subjective-surface model}
\label{sec:model}

Let $\alpha\in(0,1)$ and let $\Omega\subset\R^d$, $d\in\{2,3,4\}$, be a bounded connected image domain with $C^{4+\alpha}$ boundary. Numerical implementations on rectangular 3D and 4D grids impose the corresponding discrete boundary condition.

For $p\in\R^d$ and $\eps>0$ define
\[
        A_\eps(p)=\sqrt{\eps^2+|p|^2}.
\]
We study the initial-boundary value problem
\begin{equation}\label{eq:main-pde}
\boxed{
\begin{cases}
\partial_t u
=\nupar\Delta u+
A_\eps(\nabla u)\operatorname{div}\!\left(G(x)\dfrac{\nabla u}{A_\eps(\nabla u)}\right)
-\mu\cT_\eta(x,u),
& (t,x)\in(0,\infty)\times\Omega,\\[2mm]
u=0,& (t,x)\in(0,\infty)\times\partial\Omega,\\[1mm]
u(0,x)=u^0(x),& x\in\Omega.
\end{cases}}
\end{equation}
The boundary condition in \eqref{eq:main-pde} is the homogeneous Dirichlet condition commonly used in the subjective-surface image-segmentation setting.

Define
\begin{equation}\label{eq:operator}
\cK_{\eps,\nupar}[v]
=\nupar\Delta v+
A_\eps(\nabla v)\operatorname{div}\!\left(G(x)\frac{\nabla v}{A_\eps(\nabla v)}\right).
\end{equation}
Then \eqref{eq:main-pde} is $u_t=\cK_{\eps,\nupar}[u]-\mu\cT_\eta(x,u)$ with homogeneous boundary data.

\begin{assumption}[Data for the positive-time classical Dirichlet problem]\label{ass:data}
Assume that:
\begin{enumerate}[label=\textup{(A\arabic*)},leftmargin=2.5em]
\item $G\in C^{3+\alpha}(\overline\Omega)$ and there are constants $0<g_*\le g^*$ such that $g_*\le G(x)\le g^*$ on $\overline\Omega$;
\item $\Lambda\in C^{1+\alpha}(\overline\Omega)$, $\Lambda\ge0$, $\mu\ge0$, and $q,\eta,H,H_\eta,h_\eta$ satisfy \eqref{eq:profiles} with $0<\eta<\min\{q,1-q\}$; in particular $H_\eta'\ge0$, and this monotonicity is a structural assumption used in comparison and stability;
\item $u^0\in C^{2+\alpha}(\overline\Omega)$, $0\le u^0\le1$ in $\overline\Omega$, $u^0=0$ on $\partial\Omega$, and the first compatibility condition
\[
        \cK_{\eps,\nupar}[u^0]-\mu\cT_\eta(x,u^0)=0
        \qquad\text{on }\partial\Omega
\]
holds;
\item $\eps>0$ and $\nupar>0$ are fixed;
\end{enumerate}
\end{assumption}

Here and below, the structural data consist of the domain $\Omega$, the
dimension $d$, the exponent $\alpha$, the parameters $\eps$, $\nupar$,
$\mu$, $q$, and $\eta$, the bounds $g_*$ and $g^*$, and the prescribed
functions $G$, $\Lambda$, and $H_\eta$ together with the norms and
regularity specified in Assumption~\ref{ass:data}.

The compatibility condition included in Assumption~\ref{ass:data} is
\begin{equation}\label{eq:compatibility}
        \cK_{\eps,\nupar}[u^0](x)-\mu\cT_\eta(x,u^0(x))=0,
        \qquad x\in\partial\Omega,
\end{equation}
with derivatives taken from inside $\Omega$.

\begin{definition}[Positive-time and Schauder-classical solutions]\label{def:classical}
For $0<T<\infty$, a positive-time classical solution of \eqref{eq:main-pde} on $[0,T]$ is a function
\[
        u\in C([0,T];C^1(\overline\Omega))\cap C^{1,2}((0,T]\times\overline\Omega)
\]
that satisfies the PDE pointwise for $t>0$, attains the initial data, and satisfies $u(t,x)=0$ for $t>0$ and $x\in\partial\Omega$.  It is H\"older-classical with exponent $\beta\in(0,1)$ on $[s,T]$, $s>0$, if
\[
        u\in C^{1+\beta/2,2+\beta}([s,T]\times\overline\Omega)
\]
with $u(t,x)=0$ for $t>0$ and $x\in\partial\Omega$.  It is Schauder-classical up to the initial time with exponent $\beta$ if
\[
        u\in C^{1+\beta/2,2+\beta}([0,T]\times\overline\Omega).
\]  A global positive-time classical solution is one whose restriction to every finite time interval is positive-time classical.
\end{definition}

\begin{remark}[Parabolic H\"older convention]\label{rem:holder-convention}
Throughout the paper, $C^{1+\beta/2,2+\beta}(Q)$ denotes the standard
anisotropic parabolic H\"older space.  Thus $u$, $u_t$, $D_xu$, and
$D_x^2u$ are continuous, $u_t$ and $D_x^2u$ are H\"older continuous with
space--time exponents $(\beta/2,\beta)$, and $D_xu$ has temporal H\"older
exponent $(1+\beta)/2$ together with spatial $C^{1+\beta}$ regularity.
Equivalent standard norms may be used.  In particular, the estimate
\[
\nabla u \in C^{\beta/2,\beta}
\]
proved below is intentionally weaker than the final regularity result. Its purpose is to provide the H\"older continuity needed to apply the linear Schauder estimates.
\end{remark}

\begin{theorem}[Global classical well-posedness]\label{thm:global}
Under Assumption~\ref{ass:data}, problem \eqref{eq:main-pde} has a unique global classical solution whose restriction to every finite time interval is Schauder-classical. The solution satisfies
\begin{equation}\label{eq:range}
        0\le u(t,x)\le1,
        \qquad t\ge0,
        \quad x\in\overline\Omega.
\end{equation}
Moreover, for every $T>0$ there is a finite constant $C_T$, depending only on $T$, $\Omega$, $d$, $\eps$, $\nupar$, $g_*$, $g^*$, $\norm{G}_{C^{3+\alpha}(\overline\Omega)}$, $\norm{\Lambda}_{C^{1+\alpha}(\overline\Omega)}$, the parameters $\mu,q,\eta,H$, $\norm{u^0}_{C^{2+\alpha}}$, such that
\begin{equation}\label{eq:finite-time-c1}
        \sup_{0\le t\le T}\norm{u(t)}_{C^1(\overline\Omega)}\le C_T.
\end{equation}
For every $T>0$ there are $\beta_T\in(0,\alpha]$ and a finite constant $C_T^{\mathrm{Sch}}$ such that
\begin{equation}\label{eq:positive-holder}
        \norm{u}_{C^{1+\beta_T/2,2+\beta_T}([0,T]\times\overline\Omega)}
        \le C_T^{\mathrm{Sch}}.
\end{equation}
\end{theorem}

\begin{theorem}[Nonexpansive stability with respect to the initial data]
\label{thm:stability}
Fix $G$, $\Lambda$, $\mu$, $q$, $\eta$, $H$, $\eps$, $\nupar$, and $\Omega$
as in Assumption~\ref{ass:data}. Let $u$ and $v$ be the global
positive-time classical solutions of \eqref{eq:main-pde} associated with
two initial functions $u^0$ and $v^0$ satisfying
Assumption~\ref{ass:data}. Then, for every $t\geq 0$,
\begin{equation}\label{eq:stability}
    \norm{u(t)-v(t)}_{L^\infty(\Omega)}
    \leq
    \norm{u^0-v^0}_{L^\infty(\Omega)}.
\end{equation}
In particular, for fixed coefficient functions, the regularized flow is
nonexpansive in the supremum norm and therefore stable with respect to
perturbations of the initial data.
\end{theorem}

The stability estimate compares solutions with the same prescribed coefficient functions $G$ and $\Lambda$; only the initial data are allowed to vary.

\begin{proposition}[Variational identity for the reaction term]\label{prop:energy}
For smooth functions, $\cT_\eta$ is the $L^2$-gradient density of \eqref{eq:scalar-overlap-energy}.  Consequently, along the reaction-only subflow $u_t=-\mu\cT_\eta(x,u)$,
\begin{equation}\label{eq:energy-dissipation}
        \frac{\dd}{\dd t}\cE^\eta(u(t))
        =-\mu\int_\Omega |\cT_\eta(x,u(t,x))|^2\,\dd x\le0.
\end{equation}
For a smooth solution of the full PDE \eqref{eq:main-pde},
\[
        \frac{\dd}{\dd t}\cE^\eta(u(t))
        =\int_\Omega\cT_\eta(x,u)\cK_{\eps,\nupar}[u] \dd x
        -\mu\int_\Omega |\cT_\eta(x,u)|^2\dd x,
\]
which gives the corresponding energy-balance identity for the full image-driven flow.
\end{proposition}

\begin{lemma}[Boundary-compatible preparation of seed data]\label{lem:seed-preparation}
Let $v\in C^{2+\alpha}(\overline\Omega)$ satisfy $0\le v\le1$.  For every sufficiently small $r>0$ there exists a cutoff $\chi_r\in C^{2+\alpha}(\overline\Omega)$ such that
\[
        0\le \chi_r\le1,\qquad
        \chi_r=0 \text{ if } \operatorname{dist}(x,\partial\Omega)\le r,
        \qquad
        \chi_r=1 \text{ if } \operatorname{dist}(x,\partial\Omega)\ge 2r .
\]
Set $u_r^0=\chi_r v$. Then
$u_r^0\in C^{2+\alpha}(\overline{\Omega})$,
$0\leq u_r^0\leq 1$,
$u_r^0=0$ on $\partial\Omega$, and
$u_r^0$ satisfies the compatibility condition
\eqref{eq:compatibility}. Moreover,
\[
u_r^0=v
\quad\text{on}\quad
\left\{
x\in\Omega:
\operatorname{dist}(x,\partial\Omega)\geq 2r
\right\}.
\]
Thus, for every sufficiently small fixed $r>0$, the function $u_r^0$
provides a boundary-compatible initial datum that agrees with $v$
away from the boundary.
\end{lemma}

\begin{proof}
The existence of such cutoffs follows from the smoothness of $\partial\Omega$ and the local boundary-coordinate representation of the distance to the boundary, after the usual smoothing of the distance cutoff.  The range and boundary condition are immediate.  Since $u^0_r$ is identically zero in a neighborhood of the boundary, all first and second spatial derivatives of $u^0_r$ vanish on $\partial\Omega$ from the interior.  Hence $\cK_{\eps,\nupar}[u^0_r]=0$ on $\partial\Omega$.  Also $u^0_r=0$ on $\partial\Omega$ and $0<\eta<q$, hence $-q<-\eta$; since $H_\eta(s)=0$ for $s\le-\eta$, we have $H_\eta(u^0_r-q)=H_\eta(-q)=0$.  Therefore $\cT_\eta(x,u^0_r)=0$ on $\partial\Omega$, and \eqref{eq:compatibility} follows.
\end{proof}

\section{Structural facts}
\label{sec:structure}

The geometric operator \eqref{eq:operator} can be written in nondivergence form.  If $p\in\R^d$ and
\begin{equation}\label{eq:aij}
        a^{ij}(x,p)=\nupar\delta_{ij}+G(x)\left(\delta_{ij}-\frac{p_i p_j}{\eps^2+|p|^2}\right),
\end{equation}
then a direct calculation gives
\begin{equation}\label{eq:nondivergence}
        \cK_{\eps,\nupar}[v]
        =a^{ij}(x,\nabla v)v_{ij}+\nabla G(x)\cdot\nabla v.
\end{equation}
Here and below repeated spatial indices are summed from $1$ to $d$.

\begin{lemma}[Uniform parabolicity and global coefficient bounds]\label{lem:ellipticity}
For every $x\in\overline\Omega$ and $p,\xi\in\R^d$,
\begin{equation}\label{eq:ellipticity}
        \nupar |\xi|^2\le a^{ij}(x,p)\xi_i\xi_j\le (\nupar+g^*)|\xi|^2.
\end{equation}
Moreover, for $m=1,2$, the derivatives $D_p^m a^{ij}(x,p)$ are bounded on $\overline\Omega\times\R^d$ by constants depending only on $d$, $m$, $\eps$, and $g^*$.  The $x$-derivatives up to order two are bounded by $\norm{G}_{C^{2}(\overline\Omega)}$, and the H\"older moduli used in the positive-time estimates are controlled by $\norm{G}_{C^{3+\alpha}(\overline\Omega)}$.  In particular, these coefficient bounds do not require an a priori bound on $|p|$.
\end{lemma}

\begin{proof}
The matrix $P_\eps(p)=I-p\otimes p/(\eps^2+|p|^2)$ has eigenvalue $1$ on the subspace orthogonal to $p$ and eigenvalue $\eps^2/(\eps^2+|p|^2)$ in the direction of $p$.  Hence $0<P_\eps(p)\le I$.  Adding $\nupar I$ and using $0<G\le g^*$ gives \eqref{eq:ellipticity}.  Since $\eps>0$, the map $p\mapsto p_i p_j/(\eps^2+|p|^2)$ is smooth.  More explicitly,
\[
        \partial_{p_k}\left(\frac{p_i p_j}{\eps^2+|p|^2}\right)
        =\frac{\delta_{ik}p_j+\delta_{jk}p_i}{\eps^2+|p|^2}
        -\frac{2p_i p_j p_k}{(\eps^2+|p|^2)^2}.
\]
Each term on the right is globally bounded on $\R^d$, with constants depending only on $d$ and negative powers of $\eps$.  Differentiating once more gives rational expressions whose numerators have degree at most four and whose denominators contain powers of $\eps^2+|p|^2$ large enough to give the same global boundedness.  Hence the first and second $p$-derivatives are bounded uniformly in $p$.  Multiplication by the smooth coefficient function $G$ gives the asserted coefficient bounds.
\end{proof}

\begin{lemma}[Bounds, monotonicity, and composition estimates for the reaction term]\label{lem:reaction}
There are constants $M_\eta,L_\eta,C_\eta^{\mathrm{rxn}}<\infty$, depending only on $\norm{\Lambda}_{C^{1+\alpha}}$, $\eta$, and finitely many derivatives of $H$, such that for all $x\in\overline\Omega$ and $r,s\in\R$,
\begin{align}
        0\le \cT_\eta(x,r)&\le M_\eta,\label{eq:T-bound}\\
        |\cT_\eta(x,r)-\cT_\eta(x,s)|&\le L_\eta |r-s|,\label{eq:T-lip}\\
        \partial_r\cT_\eta(x,r)&=\Lambda(x)h_\eta(r-q)\ge0,\label{eq:T-monotone}\\
        \norm{\cT_\eta(\cdot,r)}_{C^{1+\alpha}(\overline\Omega)}&\le C_\eta^{\mathrm{rxn}}
        \quad \text{for } r\in[0,1].\label{eq:T-xbound}
\end{align}
If $0<\beta\le\alpha$ and $v\in C^\beta(\overline\Omega)$ satisfies $0\le v\le1$, then
\begin{equation}\label{eq:T-composition}
        \norm{\cT_\eta(\cdot,v(\cdot))}_{C^\beta(\overline\Omega)}
        \le C_\eta\bigl(1+\norm{v}_{C^\beta(\overline\Omega)}\bigr).
\end{equation}
Moreover,
\begin{equation}\label{eq:T-barrier}
        \cT_\eta(x,r)=0\quad\text{whenever } r\le0,
        \qquad
        \cT_\eta(x,r)\ge0\quad\text{for all } r\in\R .
\end{equation}
\end{lemma}

\begin{proof}
The estimates follow from \eqref{eq:topology-reaction}, the nonnegativity of $\Lambda$ and $H_\eta$, and the boundedness of the derivatives of $H_\eta$.  One may take
\[
        M_\eta=\norm{\Lambda}_{L^\infty},
        \qquad
        L_\eta=\norm{\Lambda}_{L^\infty}\norm{h_\eta}_{L^\infty}.
\]
Since $0<\eta<q$, if $r\le0$ then $r-q\le -q<-\eta$, and the support convention for $H_\eta$ gives $H_\eta(r-q)=0$.  Monotonicity follows from
$\partial_r\cT_\eta(x,r)=\Lambda(x)h_\eta(r-q)\ge0$.
The composition estimate follows from the product estimate in $C^\beta$ and the Lipschitz bound for $H_\eta$ on $[0,1]$:
\[
 [\Lambda H_\eta(v-q)]_{C^\beta}
 \le \norm{\Lambda}_{L^\infty}\norm{h_\eta}_{L^\infty}[v]_{C^\beta}
 +[\Lambda]_{C^\beta}.
\]
\end{proof}

\begin{lemma}[Comparison principle on a positive-time interval]\label{lem:comparison}
Let $0\le t_0<T$ and
\[
 w\in C([t_0,T];C^1(\overline\Omega))\cap C^{1,2}((t_0,T]\times\overline\Omega)
\]
satisfy
\[
 w_t-a^{ij}(t,x)w_{ij}-c^i(t,x)w_i+d(t,x)w\le0
\]
in $(t_0,T]\times\Omega$.  Assume that $a^{ij},c^i,d$ are continuous and bounded on $[t_0,T]\times\overline\Omega$, that $a^{ij}$ is uniformly parabolic, and that $d\ge0$.  If $w\le0$ at $t=t_0$ and on the spatial boundary, then $w\le0$ in $[t_0,T]\times\overline\Omega$.
\end{lemma}

\begin{proof}
For $\delta>0$ set $w_\delta=w-\delta(t-t_0)$.  If $w_\delta$ had a positive maximum at an interior point with positive time, then $w_\delta>0$, $\nabla w_\delta=0$, $D^2w_\delta\le0$, and $(w_\delta)_t\ge0$ there.  Since $d\ge0$, the differential inequality gives $0\le (w_\delta)_t\le -\delta-dw_\delta\le-\delta$, a contradiction.  Letting $\delta\downarrow0$ proves the claim.
\end{proof}

\begin{proposition}[Invariant physical range]\label{prop:range}
Every positive-time classical solution of \eqref{eq:main-pde} with $0\le u^0\le1$ and $u(t,x)=0$ for $t>0$ and $x\in\partial\Omega$ satisfies $0\le u\le1$ on its interval of existence.
\end{proposition}

\begin{proof}
We give the argument with strict barriers to avoid ambiguity at the first contact level.  Fix $\delta>0$ and $K>0$.  Since $u^0\ge0$ and the boundary condition is $u(t,x)=0$ for $t>0$ and $x\in\partial\Omega$, the function
\[
        u_\delta^-(t,x)=u(t,x)+\delta e^{Kt}
\]
is positive at the initial time and on the spatial boundary.  If $u_\delta^-$ had a first zero at an interior point $(t_0,x_0)$ with $t_0>0$, then $u(t_0,x_0)=-\delta e^{Kt_0}<0$, $\nabla u(t_0,x_0)=0$, and $D^2u(t_0,x_0)\ge0$.  By \eqref{eq:T-barrier}, $\cT_\eta(x_0,u(t_0,x_0))=0$, while uniform ellipticity gives $\cK_{\eps,\nupar}[u](t_0,x_0)\ge0$.  Hence
\[
        (u_\delta^-)_t(t_0,x_0)=u_t(t_0,x_0)+K\delta e^{Kt_0}
        \ge K\delta e^{Kt_0}>0,
\]
contradicting the first-contact condition for a nonnegative function touching zero from above.  Thus $u\ge-\delta e^{Kt}$, and letting $\delta\downarrow0$ gives $u\ge0$.

For the upper bound, the function
\[
        u_\delta^+(t,x)=u(t,x)-1-\delta e^{Kt}
\]
is negative at the initial time and on the spatial boundary.  If it had a first zero at an interior point, then $u(t_0,x_0)=1+\delta e^{Kt_0}>1$, $\nabla u(t_0,x_0)=0$, and $D^2u(t_0,x_0)\le0$.  Therefore $\cK_{\eps,\nupar}[u](t_0,x_0)\le0$, while $-\mu\cT_\eta(x_0,u(t_0,x_0))\le0$.  Hence
\[
        (u_\delta^+)_t(t_0,x_0)=u_t(t_0,x_0)-K\delta e^{Kt_0}<0,
\]
which contradicts the first-contact condition for a nonpositive function touching zero from below.  Thus $u\le1+\delta e^{Kt}$, and letting $\delta\downarrow0$ gives $u\le1$.
\end{proof}

\section{Local solvability and stability}
\label{sec:local}

\begin{proposition}[Compatible local quasilinear Dirichlet theorem]
\label{prop:external-local}
Let $\Omega$ have $C^{2+\sigma}$ boundary, $0<\sigma<1$.  Consider
\[
 \begin{cases}
 w_t-A^{ij}(x,t,w,Dw)w_{ij}=f(x,t,w,Dw)
     &\text{in }(0,S)\times\Omega,\\
 w=0&\text{on }(0,S)\times\partial\Omega,\\
 w(0,\cdot)=w^0&\text{in }\Omega.
 \end{cases}
\]
Suppose that, on every bounded $(z,p)$-set, $A^{ij}$ is symmetric and
uniformly parabolic; the maps $A^{ij}$, $\partial_zA^{ij}$, $\partial_{p_k}A^{ij}$, $f$,
$\partial_zf$, and $\partial_{p_k}f$ are jointly H\"older continuous in
$(t,x,z,p)$ on every bounded $(z,p)$-set; and $A^{ij}$ and $f$ are
locally Lipschitz in $(z,p)$, uniformly in $(t,x)$.  Assume
$w^0\in C^{2+\sigma}(\overline\Omega)$ and
\[
 w^0=0,\qquad
 A^{ij}(x,0,w^0,Dw^0)w^0_{ij}+f(x,0,w^0,Dw^0)=0
 \quad\text{on }\partial\Omega.
\]
Then, for some $S>0$ and $\beta\in(0,\sigma]$, the problem has a unique
solution
\[
 w\in C^{1+\beta/2,2+\beta}([0,S]\times\overline\Omega).
\]

\vspace{.1in}
This proposition is the standard consequence of linear Dirichlet Schauder
solvability and the quasilinear contraction/fixed-point construction in a
bounded $C^{1+\beta/2,2+\beta}$ neighborhood of the initial datum; see
\cite[Chapter~V, Sections~6--7]{LSU1968},
\cite[Chapter~6]{Friedman1964}, and
\cite[Chapters~XIII--XIV]{Lieberman1996}.
\end{proposition}

\begin{lemma}[Short-time compatible Schauder solvability]
\label{lem:local}
Let $\gamma\in(0,\alpha]$. Assume {\rm (A1)}, {\rm (A2)}, and {\rm (A4)}, and let
$w^0\in C^{2+\gamma}(\overline{\Omega})$ satisfy
\[
w^0=0,
\qquad
\mathcal{K}_{\varepsilon,\nu}[w^0]
-\mu\cT_\eta(x,w^0)=0
\quad\text{on }\partial\Omega.
\]
Then there exist $S_\gamma>0$ and $\beta\in(0,\gamma]$ such that
\eqref{eq:main-pde}, with initial value $w^0$, has a unique solution
\[
u\in
C^{1+\beta/2,\,2+\beta}
\bigl([0,S_\gamma]\times\overline{\Omega}\bigr).
\]
In particular,
\[
u\in
C\bigl([0,S_\gamma];C^1(\overline{\Omega})\bigr)
\cap
C^{1,2}\bigl((0,S_\gamma]\times\overline{\Omega}\bigr).
\]
\end{lemma}

\begin{proof}
We apply Proposition~\ref{prop:external-local}. In the notation of that
proposition, set
\[
A^{ij}(x,t,z,p)=a^{ij}(x,p),
\qquad
f(x,t,z,p)=B(x,z,p),
\qquad
\varphi=0.
\]

Write the equation as
\[
u_t-a^{ij}(x,\nabla u)u_{ij}=B(x,u,\nabla u),
\]
where
\[
B(x,r,p)
=
\nabla G(x)\cdot p
-
\mu\cT_\eta(x,r).
\]

By Lemma~\ref{lem:ellipticity}, the matrix $a^{ij}(x,p)$ is symmetric and
satisfies
\[
\nu |\xi|^2
\leq
a^{ij}(x,p)\xi_i\xi_j
\leq
(\nu+g^*)|\xi|^2
\]
for every
\[
x\in\overline{\Omega},
\qquad
p,\xi\in\mathbb{R}^d.
\]
Moreover, $a^{ij}$ is $C^{1+\alpha}$ in $x$, smooth in $p$, and its first two
derivatives with respect to $p$ are bounded on every bounded $p$-set.

The function $B$ is $C^\alpha$ in $x$ and locally Lipschitz continuous in
$(r,p)$, uniformly on bounded sets, because
\[
G\in C^{3+\alpha}(\overline{\Omega}),
\qquad
\Lambda\in C^{1+\alpha}(\overline{\Omega}),
\]
and $H_\eta$ is smooth. The boundary is of class $C^{4+\alpha}$, and the
lateral boundary datum is zero. The identity
\[
\mathcal{K}_{\varepsilon,\nu}[w^0]
-\mu\cT_\eta(x,w^0)=0
\quad\text{on }\partial\Omega
\]
is precisely the first compatibility condition at the corner
$\{0\}\times\partial\Omega$.

Proposition~\ref{prop:external-local} therefore yields
$S_\gamma>0$, $\beta\in(0,\gamma]$, and a solution
\[
u\in
C^{1+\beta/2,\,2+\beta}
\bigl([0,S_\gamma]\times\overline{\Omega}\bigr).
\]
Uniqueness in this class follows from the quasilinear uniqueness statement
in Proposition~\ref{prop:external-local}; it also follows from
Lemma~\ref{lem:stability-local}.

Thus, the result provides local classical solvability for compatible initial
data. The global solution is subsequently obtained from the a priori
estimates and continuation argument developed below.
\end{proof}

\begin{lemma}[Supremum-norm stability on a common interval]\label{lem:stability-local}
Let $u,v\in C([0,T];C^1(\overline\Omega))\cap C^{1,2}((0,T]\times\overline\Omega)$ be two positive-time classical solutions of \eqref{eq:main-pde} on $[0,T]$ with the same coefficient functions $G$ and $\Lambda$.  Then
\[
        \norm{u(t)-v(t)}_{L^\infty(\Omega)}
        \le \norm{u(0)-v(0)}_{L^\infty(\Omega)},
        \qquad 0\le t\le T.
\]
\end{lemma}

\begin{proof}
Set $z=u-v$.  For $\lambda\in[0,1]$ let $v_\lambda=v+\lambda(u-v)$.  By the mean-value formula, $z$ satisfies
\begin{equation}\label{eq:linearized-difference}
        z_t=\bar a^{ij}(t,x)z_{ij}+\bar c^i(t,x)z_i-
        \mu\bar r(t,x)z,
\end{equation}
where
\[
        \bar a^{ij}(t,x)=\int_0^1 a^{ij}(x,\nabla v_\lambda(t,x))\,\dd\lambda,
\]
\[
        \bar c^k(t,x)=\partial_{x_k}G(x)+
        \int_0^1
        \frac{\partial a^{ij}}{\partial p_k}(x,\nabla v_\lambda(t,x))
        (v_\lambda)_{ij}(t,x)\,\dd\lambda,
\]
and
\[
        \bar r(t,x)=\int_0^1 \partial_s\cT_\eta(x,v_\lambda(t,x))\,\dd\lambda.
\]
For every $\tau>0$, the positive-time classical regularity of $u$ and $v$ gives boundedness of $D^2u$ and $D^2v$ on $[\tau,T]\times\overline\Omega$.  Since the first $p$-derivatives of $a^{ij}$ are globally bounded by Lemma~\ref{lem:ellipticity}, the coefficients $\bar c^i$ are bounded on $[\tau,T]$.  On the same interval \eqref{eq:linearized-difference} is uniformly parabolic by Lemma~\ref{lem:ellipticity}.  Lemma~\ref{lem:reaction} gives $\bar r\ge0$.

Let $M=\norm{z(0)}_{L^\infty(\Omega)}$.  Fix $\delta>0$.  By continuity at $t=0$, choose
\[
 0<\tau_\delta<\min\{\delta,T\}
 \qquad\text{and}\qquad
 \norm{z(\tau_\delta)}_{L^\infty(\Omega)}\le M+\delta.
\]
On $[\tau_\delta,T]\times\Omega$, define
\[
\mathcal{L}w
:=
w_t-\bar a^{ij}w_{ij}-\bar c^iw_i+\mu\bar r\,w.
\]
All coefficients of $\mathcal L$ are bounded there, the principal matrix is
uniformly parabolic, and $\bar r\ge0$.  Since $\mathcal Lz=0$,
\[
 \mathcal L\bigl(z-(M+\delta)\bigr)
 =-\mu\bar r(M+\delta)\le0,
\]
and likewise
\[
 \mathcal L\bigl(-z-(M+\delta)\bigr)
 =-\mu\bar r(M+\delta)\le0.
\]
Both comparison functions are nonpositive on the lateral boundary and at
$t=\tau_\delta$.  Lemma~\ref{lem:comparison} therefore gives
\[
 \norm{z(t)}_{L^\infty(\Omega)}\le M+\delta,
 \qquad \tau_\delta\le t\le T.
\]
For each fixed $t>0$, one has $\tau_\delta<t$ for all sufficiently small
$\delta$.  Letting $\delta\downarrow0$ yields
$\norm{z(t)}_{L^\infty(\Omega)}\le M$ for every $t>0$, and continuity at
$t=0$ completes the proof on $[0,T]$.
\end{proof}

\section{Finite-time estimates and continuation}
\label{sec:continuation}

For each prescribed finite time \(T>0\), the estimates hold with constants
depending only on \(T\) and the fixed data. These bounds provide the control
required for the continuation argument.

We first prove the finite-time gradient estimate needed for continuation.  The proof has two parts.  Gradient estimates near $\partial\Omega$ control the normal derivative on $(0,T]\times\partial\Omega$, and interior gradient estimates then prevent the spatial gradient from attaining a larger value in the interior.

Here and below, $Q'\Subset Q$ means that $\overline{Q'}$ is a compact
subset of $Q$.
\begin{lemma}[Interior regularity improvement for positive times]
\label{lem:interior-bootstrap}
Let $u$ be a positive-time classical solution of
\eqref{eq:main-pde}. For every pair of parabolic cylinders
\[
Q'\Subset Q\Subset (0,T]\times\Omega,
\]
there exists $\sigma\in(0,\alpha]$ such that
\[
D_xu\in C^{1+\sigma/2,\,2+\sigma}(Q').
\]
Equivalently,
\[
u\in C^{1+\sigma/2,\,3+\sigma}(Q').
\]
Consequently, all derivatives required to differentiate the equation once
with respect to the spatial variables exist and are H\"older continuous at
every positive-time interior point.
\end{lemma}

This interior regularity improvement justifies the spatial differentiation
used in the Bernstein gradient estimate.
\begin{proof}
Choose $Q''$ with $Q'\Subset Q''\Subset Q$.  Positive-time local
regularity gives $u\in C^{1+\beta/2,2+\beta}(Q)$ for some $\beta>0$.
For a coordinate direction $e_k$ set
\[
 v_h(t,x)=\frac{u(t,x+he_k)-u(t,x)}{h}.
\]
Writing $u^h(t,x)=u(t,x+he_k)$ and subtracting the equations at
$(t,x+he_k)$ and $(t,x)$ gives
\[
 (v_h)_t-A_h^{ij}(v_h)_{ij}-B_h^\ell(v_h)_\ell-C_hv_h=f_h,
\]
where
\[
 A_h^{ij}=a^{ij}(x+he_k,\nabla u^h),
\]
\[
 B_h^\ell=
 \partial_{x_\ell}G(x+he_k)
 +\int_0^1\partial_{p_\ell}a^{ij}
 \bigl(x+he_k,\nabla u+s(\nabla u^h-\nabla u)\bigr)\,ds\;u_{ij},
\]
\[
 C_h=-\mu\Lambda(x+he_k)
 \int_0^1H_\eta'\bigl(u+s(u^h-u)-q\bigr)\,ds,
\]
and
\[
\begin{aligned}
 f_h={}&
 \frac{a^{ij}(x+he_k,\nabla u)-a^{ij}(x,\nabla u)}{h}\,u_{ij}\\
 &+\frac{\partial_{x_i}G(x+he_k)-\partial_{x_i}G(x)}{h}\,u_i
 -\mu\frac{\Lambda(x+he_k)-\Lambda(x)}{h}H_\eta(u-q).
\end{aligned}
\]
All unshifted functions on the right are evaluated at $(t,x)$.  The
term $\partial_{x_\ell}G(x+he_k)(v_h)_\ell$ in $B_h^\ell(v_h)_\ell$
is the difference quotient of the original drift $\nabla G\cdot\nabla u$;
the remaining drift difference is the second term in $f_h$.  Choose
$0<\sigma<\min\{\alpha,\beta\}$.  The composition and product estimates
in parabolic H\"older spaces give, uniformly for all sufficiently small $h$,
\[
 \|A_h\|_{C^{\sigma/2,\sigma}(Q'')}
 \le C\bigl(1+\|u\|_{C^{1+\beta/2,2+\beta}(Q)}\bigr),
\]
\[
 \|B_h\|_{C^{\sigma/2,\sigma}(Q'')}
 \le C\bigl(1+\|D^2u\|_{C^{\beta/2,\beta}(Q)}\bigr),
 \qquad
 \|C_h\|_{C^{\sigma/2,\sigma}(Q'')}\le C,
\]
and
\[
 \|f_h\|_{C^{\sigma/2,\sigma}(Q'')}
 \le C\bigl(1+\|u\|_{C^{1+\beta/2,2+\beta}(Q)}\bigr).
\]
Here the difference quotients of $a^{ij}$, $\nabla G$, and $\Lambda$ are
controlled by their $C^{1+\alpha}$ norms.  More explicitly, on the bounded
positive-time gradient range,
\[
 \bigl|D_pa(x_1,p_1)-D_pa(x_2,p_2)\bigr|
 \le \|D^2_{xp}a\|_{L^\infty}|x_1-x_2|
      +\|D^2_{pp}a\|_{L^\infty}|p_1-p_2|.
\]
Applying this inequality with
$p_j=\nabla u(t_j,x_j)+s(\nabla u^h-\nabla u)(t_j,x_j)$,
using the positive-time parabolic H\"older norm of $\nabla u$, and then the
product estimate with $D^2u$, gives the displayed bound for $B_h$ uniformly
in $h$.  The same mean-value and product estimates control $A_h$, $C_h$, and
$f_h$.  Thus the constant is independent of $h$.  Moreover $A_h$ has
ellipticity constant $\nupar$.
Interior linear Schauder estimates therefore give, for some
$\sigma\in(0,\min\{\alpha,\beta\}]$,
\[
 \|v_h\|_{C^{1+\sigma/2,2+\sigma}(Q')}
 \le C,
\]
with $C$ independent of $h$. Choose $0<\sigma'<\sigma$. The uniform
estimate above and the compactness of the embedding of
\[
C^{1+\sigma/2,2+\sigma}(Q')
\]
into
\[
C^{1+\sigma'/2,2+\sigma'}(Q')
\]
give a sequence $h_j\to0$ and a function $v$ such that
\[
v_{h_j}\longrightarrow v
\quad\text{in}\quad
C^{1+\sigma'/2,2+\sigma'}(Q').
\]
Since the difference quotients $v_{h_j}$ converge to $u_{x_k}$, the limit satisfies
\[
v=u_{x_k}.
\]
Therefore,
\[
u_{x_k}\in C^{1+\sigma'/2,\,2+\sigma'}(Q').
\]
Since $k$ is arbitrary,
\[
D_xu\in C^{1+\sigma'/2,\,2+\sigma'}(Q').
\]
Renaming $\sigma'$ as $\sigma$ gives the asserted interior regularity.
\end{proof}

\begin{proposition}[Finite-time gradient estimate]\label{prop:finite-gradient}
Let $U$ be a positive-time classical solution of \eqref{eq:main-pde} on $[0,S]$ with $U(t,x)=0$ for $t>0$ and $x\in\partial\Omega$, initial datum $U^0\in C^{2+\alpha}(\overline\Omega)$ satisfying $U^0=0$ on $\partial\Omega$, $U(t)\to U^0$ in $C^1(\overline\Omega)$ as $t\downarrow0$, and $0\le U\le1$.  Then there exists a constant
\[
        C_{\nabla}=C_{\nabla}\bigl(S,\Omega,d,\eps,\nupar,g_*,g^*,
        \norm{G}_{C^{3+\alpha}},\norm{\Lambda}_{C^{1+\alpha}},
        \mu,q,\eta,H,\norm{U^0}_{C^{2+\alpha}}\bigr)
\]
such that
\[
        \norm{\nabla U}_{L^\infty((0,S)\times\Omega)}\le C_{\nabla}.
\]
\end{proposition}

\begin{proof}
Write the equation in nondivergence form as
\begin{equation}\label{eq:tailored-gradient-pde}
        U_t=a^{ij}(x,\nabla U)U_{ij}+b^i(x)U_i+F(x,U),
\end{equation}
where
\[
        b=\nabla G,
        \qquad
        F(x,r)=-\mu\Lambda(x)H_\eta(r-q),
\]
and
\[
        a^{ij}(x,p)=\nupar\delta_{ij}
        +G(x)\left(\delta_{ij}-\frac{p_i p_j}{\eps^2+|p|^2}\right).
\]
By Lemma~\ref{lem:ellipticity},
\[
        \nupar|\xi|^2\le a^{ij}(x,p)\xi_i\xi_j
        \le (\nupar+g^*)|\xi|^2.
\]
Moreover $D_xa^{ij}$ is bounded independently of $p$, $D_p a^{ij}$ is bounded because $\eps>0$, $b\in C^{2+\alpha}(\overline\Omega)$, and $F_x,F_r$ are bounded on $\overline\Omega\times[0,1]$, with
\[
        F_r(x,r)=-\mu\Lambda(x)H_\eta'(r-q)\le0.
\]

We first control the boundary gradient.  Since $\partial\Omega$ is $C^{4+\alpha}$, there are $\delta_0>0$ and $M_\rho>0$ such that the interior distance function
\[
        \rho(x)=\operatorname{dist}(x,\partial\Omega)
\]
is $C^2$ in the neighborhood of the boundary $\Omega_{\delta_0}=\{x\in\Omega:0<\rho(x)<\delta_0\}$ and
\[
        |D^2\rho|\le M_\rho \quad\text{there}.
\]
Let $n=\nabla\rho$ be the inward unit normal.  Choose $\kappa>0$ so large that
\[
        \kappa\nupar
        \ge
        1+d(\nupar+g^*)M_\rho+\norm{\nabla G}_{L^\infty(\Omega)}.
\]
For constants $M>0$ to be fixed, set
\[
        W(x)=M\bigl(1-e^{-\kappa\rho(x)}\bigr).
\]
Increasing $M$ if necessary, we may arrange that
\[
        W\ge1 \quad\text{on } \{\rho=\delta_0\},
        \qquad
        W\ge U^0 \quad\text{in } \Omega_{\delta_0}.
\]
The second inequality is possible because $U^0=0$ on $\partial\Omega$ and $U^0\in C^1(\overline\Omega)$; equivalently, the quotient
\[
        \frac{U^0(x)}{1-e^{-\kappa\rho(x)}}
\]
extends boundedly to the boundary.

In the neighborhood of the boundary,
\[
        \nabla W=M\kappa e^{-\kappa\rho}n,
        \qquad
        D^2W=-M\kappa^2 e^{-\kappa\rho} n\otimes n
        +M\kappa e^{-\kappa\rho}D^2\rho.
\]
Equivalently, if $\phi(s)=M(1-e^{-\kappa s})$, then $\phi'>0$ and $\phi''=-\kappa\phi'$, so
\[
\begin{aligned}
        &W_t-a^{ij}(x,\nabla W)W_{ij}-b(x)\cdot\nabla W
        +\mu\Lambda(x)H_\eta(W-q) \\
        &\qquad
        =-a^{ij}(x,\phi'n)\bigl(\phi'' n_i n_j+\phi'\rho_{ij}\bigr)
        -\phi'\nabla G\cdot n
        +\mu\Lambda(x)H_\eta(W-q) \\
        &\qquad
        \ge
        \phi'\bigl(\kappa\nupar-d(\nupar+g^*)M_\rho-\norm{\nabla G}_{L^\infty}\bigr)
        \ge0.
\end{aligned}
\]
Thus $W$ is a supersolution in the neighborhood of the boundary.  We compare it with $U$ by applying Lemma~\ref{lem:comparison} to the linearized difference $z=U-W$.  Indeed,
\[
\begin{aligned}
        &a^{ij}(x,\nabla U)U_{ij}-a^{ij}(x,\nabla W)W_{ij} \\
        &\qquad =a^{ij}(x,\nabla U)z_{ij}
        +\left(\int_0^1 \partial_{p_k}a^{ij}\bigl(x,\nabla W+s\nabla z\bigr)\,ds\right)W_{ij}z_k,
\end{aligned}
\]
and
\[
        \cT_\eta(x,U)-\cT_\eta(x,W)
        =\left(\int_0^1 \partial_r\cT_\eta(x,W+s z)\,ds\right)z .
\]
Equivalently, $z$ satisfies an inequality of the form
\[
        z_t-\tilde a^{ij}z_{ij}-\tilde c^k z_k+\tilde d z\le0,
\]
where $\tilde a^{ij}=a^{ij}(x,\nabla U)$ is uniformly elliptic and $\tilde c^k$ is bounded in the boundary neighborhood.  Indeed, this bound uses only the global bound for $D_pa$ and the fixed $C^2$ norm of $W$, and hence does not depend on an a priori bound for $\nabla U$.  Moreover,
\[
        \tilde d
        =\mu\int_0^1\Lambda(x)H_\eta'(W+s z-q)\,\dd s\ge0.
\]
This explicit sign follows exactly from the monotonicity of $r\mapsto H_\eta(r-q)$.  To avoid using the equation at the initial corner $t=0$, fix $\tau>0$ and set
\[
 \varepsilon_\tau=\bigl\|(U(\tau)-W)_+\bigr\|_{L^\infty(\Omega_{\delta_0})}.
\]
Because $U(t)\to U^0$ in $C^1(\overline\Omega)$ and $U^0\le W$, one has $\varepsilon_\tau\to0$.  Since adding a nonnegative constant leaves the principal part unchanged and increases the nondecreasing reaction, $W+\varepsilon_\tau$ is again a supersolution.  On the parabolic boundary of $[\tau,S]\times\Omega_{\delta_0}$ we have
$U(\tau)\le W+\varepsilon_\tau$, $U=0\le W+\varepsilon_\tau$ on $\{\rho=0\}$, and $U\le1\le W+\varepsilon_\tau$ on $\{\rho=\delta_0\}$.  Lemma~\ref{lem:comparison} therefore yields
\[
        U(t,x)\le W(x)+\varepsilon_\tau,
        \qquad \tau\le t\le S,
        \quad 0\le\rho(x)\le\delta_0.
\]
Letting $\tau\downarrow0$ and using continuity gives
\[
        0\le U(t,x)\le W(x),
        \qquad 0\le t\le S,
        \quad 0\le\rho(x)\le\delta_0.
\]
On $\partial\Omega$ the tangential derivatives of $U$ vanish because $u(t,x)=0$ for all $x\in\partial\Omega$ and positive times.  The preceding inequality and $U\ge0$ imply
\[
        0\le \partial_n U(t,x)\le \partial_n W(x)=M\kappa,
        \qquad x\in\partial\Omega,
        \quad 0<t<S.
\]
Therefore
\begin{equation}\label{eq:boundary-gradient-bound}
        |\nabla U(t,x)|\le M\kappa,
        \qquad x\in\partial\Omega,
        \quad 0<t<S.
\end{equation}

It remains to exclude a larger interior maximum.  Put
\[
        w=\frac12|\nabla U|^2.
\]

Fix $\tau>0$. By Lemma~\ref{lem:interior-bootstrap}, on every compact
positive-time interior cylinder the solution has the spatial derivatives
needed to differentiate the equation once and to evaluate the resulting
identity pointwise. This is sufficient because the spatial boundary values
of $|\nabla U|$ have already been controlled by
\eqref{eq:boundary-gradient-bound}. Hence, any larger maximum on
\[
[\tau,S']\times\overline{\Omega},
\qquad S'<S,
\]
is attained at a spatially interior point, where the interior regularity
result applies. If such a maximum lies on the terminal slice $t=S'$, the
time derivative is interpreted from the left; at a terminal-time maximum
one still has $(Y_\delta)_t\geq0$. The calculation is therefore classical
on each such finite time cylinder, and the final inequality depends only
on the structural constants displayed below, rather than on any higher
norm of $U$. At any interior point where $\nabla w=0$, differentiating \eqref{eq:tailored-gradient-pde}, multiplying by $U_k$, and summing over $k$ gives
\[
\begin{aligned}
        w_t-a^{ij}(x,\nabla U)w_{ij}
        &=U_k\,\partial_{x_k}a^{ij}(x,\nabla U)U_{ij}
        +U_k\,\partial_{p_\ell}a^{ij}(x,\nabla U)U_{\ell k}U_{ij} \\
        &\quad+U_k\,\partial_{x_k}b^i(x)U_i
        +U_k b^i(x)U_{ik}
        +U_kF_{x_k}(x,U)+F_r(x,U)|\nabla U|^2 \\
        &\quad-a^{ij}(x,\nabla U)U_{ki}U_{kj}.
\end{aligned}
\]
The identities $w_\ell=U_kU_{k\ell}=0$ imply, for every $\ell$ and every $i$, that
\[
        U_kU_{k\ell}=0,
        \qquad
        U_kU_{ki}=0.
\]
Consequently the term containing $\partial_{p_\ell}a^{ij}$ and the transport term $U_kb^iU_{ik}$ vanish at such a point.  At such a point, ellipticity and Young's inequality give
\[
 -a^{ij}U_{ki}U_{kj}\le-\nupar|D^2U|^2,
\]
while
\[
 |U_k\partial_{x_k}a^{ij}U_{ij}|
 \le C|\nabla U||D^2U|
 \le \frac{\nupar}{4}|D^2U|^2+C_\nu|\nabla U|^2.
\]
Furthermore,
\[
 |U_k(\partial_{x_k}b^i)U_i|\le C|\nabla U|^2,
 \qquad
 |U_kF_{x_k}(x,U)|\le C|\nabla U|\le C(1+w),
\]
and $F_r|\nabla U|^2\le0$.  Therefore
\[
        w_t-a^{ij}(x,\nabla U)w_{ij}
        \le
        -\frac{3\nupar}{4}|D^2U|^2+C_0(1+w)
        \le C_0(1+w)
\]
at every interior point where $\nabla w=0$.  Here $C_0$ depends only on the data listed in the statement.

Choose $C>C_0$ and define
\[
        Y(t,x)=e^{-Ct}(w(t,x)+1).
\]
For $\delta>0$ set $Y_\delta=Y-\delta(t-\tau)$.  If $Y_\delta$ attained a positive maximum at an interior point with time larger than $\tau$, then $\nabla w=0$ there and
\[
\begin{aligned}
        0&\le (Y_\delta)_t-a^{ij}(x,\nabla U)(Y_\delta)_{ij}\\
        &=e^{-Ct}\bigl(w_t-a^{ij}(x,\nabla U)w_{ij}-C(w+1)\bigr)-\delta<0,
\end{aligned}
\]
a contradiction.  Thus, on $[\tau,S']\times\overline\Omega$, the maximum of $Y_\delta$
lies on the initial slice $t=\tau$ or on the spatial boundary.  Therefore, after letting $\delta\downarrow0$, the maximum of $Y$ on
$[\tau,S']\times\overline{\Omega}$ is attained either on the initial
slice $t=\tau$ or on the spatial boundary.
Since the resulting bound is independent of $S'<S$, we may let $S'\uparrow S$.  Since $U(t)\to U^0$ in $C^1(\overline\Omega)$ as $t\downarrow0$, the contribution from the time slice $t=\tau$ converges to $\frac12\norm{\nabla U^0}_{L^\infty}^2$ as $\tau\downarrow0$.  Combining this with \eqref{eq:boundary-gradient-bound} and then letting $\tau\downarrow0$ gives
\[
        \sup_{(0,S)\times\Omega}w
        \le
        e^{CS}\left(1+\frac12\max\left\{
        \norm{\nabla U^0}_{L^\infty(\Omega)}^2,
        M^2\kappa^2
        \right\}\right)-1.
\]
Since \(w=\frac12|\nabla U|^2\), the desired gradient bound follows.
\end{proof}

\begin{lemma}[Finite-time $C^1$ bound]\label{lem:gradient}
Let $u$ be a positive-time classical solution of \eqref{eq:main-pde} on $[0,T_*)$, where $0<T_*\le\infty$, and assume $0\le u\le1$.  For every finite $S>0$ there exists a constant $C_S<\infty$, depending only on $S$, $\Omega$, $d$, $\eps$, $\nupar$, $g_*$, $g^*$, $\norm{G}_{C^{3+\alpha}(\overline\Omega)}$, $\norm{\Lambda}_{C^{1+\alpha}(\overline\Omega)}$, the parameters $\mu,q,\eta,H$, and $\norm{u^0}_{C^{2+\alpha}}$, such that
\begin{equation}\label{eq:gradient-bound}
        \sup_{0\le t<\min\{S,T_*\}}\norm{u(t)}_{C^1(\overline\Omega)}\le C_S.
\end{equation}
\end{lemma}

\begin{proof}
The \(C^0\) part follows from the range estimate. Fix
\(S'<\min\{S,T_*\}\). On \([0,S']\), the solution satisfies
\(u(t,x)=0\) for \(t>0\) and \(x\in\partial\Omega\), its initial datum
satisfies the boundary condition, and \(0\leq u\leq1\) by
Proposition~\ref{prop:range}. Proposition~\ref{prop:finite-gradient}
therefore gives
\[
\norm{\nabla u}_{L^\infty((0,S')\times\Omega)}\leq C_S,
\]
where \(C_S\) depends only on the data listed in the statement and is
valid for every \(S'<\min\{S,T_*\}\). Letting
\(S'\uparrow\min\{S,T_*\}\) gives \eqref{eq:gradient-bound}.
\end{proof}

\begin{lemma}[H\"older-gradient estimate]
\label{lem:quoted-gradient}
Fix $0<\tau_0\leq T$. For each $T'\in[\tau_0,T]$, let $z$ be an already
existing classical solution of
\[
z_t-A^{ij}(X,z,Dz)z_{ij}=F(X,z,Dz)
\qquad\text{in }Q_{T'}:=(0,T')\times\Omega,
\]
with prescribed data on the parabolic boundary
\[
\partial_pQ_{T'}
=
(\{0\}\times\overline{\Omega})
\cup
([0,T']\times\partial\Omega).
\]
Assume the hypotheses of \cite[Theorem~4.7]{Lieberman1986} with
source exponent $\gamma=2$: for some $K,\lambda_K,\mu_K>0$,
\begin{align*}
|z|+|Dz|&\leq K,\\
A^{ij}(X,z,Dz)\xi_i\xi_j
&\geq \lambda_K|\xi|^2,
\qquad \xi\in\mathbb{R}^d,
\end{align*}
and, in the notation of that theorem,
\[
|A^{ij}_z|+|A^{ij}_p|
+(d^*)^{\gamma-2}
\bigl(
|A^{ij}_x|+|A^{ij}_t|+|F|
\bigr)
\leq \mu_K
\quad\text{in }Q_{T'}.
\]
Assume also that $\partial Q_{T'}\in H_\gamma$ and that the prescribed
parabolic-boundary datum belongs to
$H_\gamma(\partial_pQ_{T'})$. Then there exist
$\beta\in(0,1)$ and $C<\infty$ such that
\[
\|Dz\|_{C^{\beta/2,\beta}(\overline{Q}_{T'})}
\leq C
\qquad\text{for every }T'\in[\tau_0,T].
\]
The exponent $\beta$ and the constant $C$ may be chosen uniformly for
$T'\in[\tau_0,T]$ whenever the structural quantities, the bounds for
$z$ and $Dz$, and the parabolic-boundary data are uniformly controlled.
\end{lemma}

In the application below, the estimate depends only on the quantities
listed in Lemma~\ref{lem:quoted-gradient} and requires no prior H\"older
bound for $D^2z$.

\begin{lemma}[Uniform gradient H\"older estimate]\label{lem:gradient-holder}
Fix $0<\tau_0\le T$.  Let $T'\in[\tau_0,T]$, and let $u$ be a
Schauder-classical solution of \eqref{eq:main-pde} on $[0,T']$ with the fixed
initial datum $u^0$ from Assumption~\ref{ass:data}.  Assume
\[
 \sup_{0\le t\le T'}\|u(t)\|_{C^1(\overline\Omega)}\le M_1.
\]
Then there are $\beta\in(0,\alpha]$ and $C<\infty$, depending only on
$\tau_0$, $T$, $M_1$, the fixed initial and boundary data, and the
structural data, such that
\[
 \|\nabla u\|_{C^{\beta/2,\beta}([0,T']\times\overline\Omega)}\le C.
\]
The exponent and constant are independent of the particular
$T'\in[\tau_0,T]$.
\end{lemma}

\begin{proof}
We verify the hypotheses of Lemma~\ref{lem:quoted-gradient} for
\eqref{eq:main-pde}. Write
\[
u_t-a^{ij}(x,\nabla u)u_{ij}=B(x,u,\nabla u),
\qquad
B(x,r,p)=\nabla G(x)\cdot p-\mu\Lambda(x)H_\eta(r-q).
\]
The range estimate and the assumed $C^1$ bound give
\[
|u|+|\nabla u|\leq 1+M_1.
\]
Lemma~\ref{lem:ellipticity} gives the common ellipticity bounds
\[
\nupar|\xi|^2
\leq
a^{ij}(x,p)\xi_i\xi_j
\leq
(\nupar+g^*)|\xi|^2.
\]

On the range
\[
|r|\leq 1,
\qquad
|p|\leq M_1,
\]
all coefficient quantities required by conditions (4.17) and (4.25) of
\cite{Lieberman1986} are bounded by constants depending only on $M_1$
and the structural data. Indeed, $a^{ij}$ is independent of $t$ and $r$,
smooth in $p$, and has the required bounded derivatives with respect to
$x$ and $p$ on this range. The lower-order term satisfies
\[
|B(x,r,p)|\leq C(1+|p|),
\]
and its required derivatives are bounded. In particular,
\[
B_p=\nabla G,
\qquad
B_{pp}=B_{rp}=0,
\]
while $B_r$, $B_{rr}$, and $D_xB$ are controlled by the fixed norms of
$G$, $\Lambda$, and $H_\eta$.

We apply the cited theorem with the admissible endpoint exponent
$\gamma=2$. In Lieberman's notation, the principal coefficients are
\[
a_{\mathrm L}^{ij}(X,z,p)=a^{ij}(x,p),
\]
and the lower-order coefficient is
\[
a_{\mathrm L}(X,z,p)=B(x,z,p).
\]
Thus,
\[
\partial_z a_{\mathrm L}^{ij}
=
\partial_t a_{\mathrm L}^{ij}
=
0.
\]
In condition (4.25), the factor $(d^*)^{\gamma-2}$ is equal to one, and
\[
|\partial_z a_{\mathrm L}^{ij}|
+
|D_pa_{\mathrm L}^{ij}|
+
|D_xa_{\mathrm L}^{ij}|
+
|\partial_ta_{\mathrm L}^{ij}|
+
|a_{\mathrm L}|
\leq C(M_1)
\]
on the bounded range determined by
\[
|u|\leq 1,
\qquad
|\nabla u|\leq M_1.
\]

We next verify that the prescribed boundary data satisfy the required
regularity with a bound that is uniform for $T'\in[\tau_0,T]$. Rescale
time by
\[
s=\frac{t}{T'}
\]
so that the problem is transferred to the fixed cylinder
$(0,1)\times\Omega$. Since $\Omega$ has $C^{4+\alpha}$ boundary, there
exist an open neighborhood $U\supset\overline{\Omega}$ and a bounded
linear extension operator
\[
E:
C^{2+\alpha}(\overline{\Omega})
\longrightarrow
C^{2+\alpha}(U).
\]
Choose an open interval $J\supset[0,1]$ and a function
$\widetilde{\zeta}\in C^\infty(J)$ satisfying
\[
\widetilde{\zeta}(0)=1,
\]
and define
\[
\widetilde{\Phi}(s,x)
=
\widetilde{\zeta}(s)\,Eu^0(x),
\qquad
(s,x)\in J\times U.
\]
Because $u^0=0$ on $\partial\Omega$, the restriction of
$\widetilde{\Phi}$ to the parabolic boundary of
$(0,1)\times\Omega$ coincides with the rescaled boundary data: it equals
$u^0$ on $\{0\}\times\overline{\Omega}$ and vanishes on
$[0,1]\times\partial\Omega$. In the notation of
\cite[p.~350]{Lieberman1986}, the space $H_2(\partial_pQ)$ consists of
traces of $H_2$ functions defined in a neighborhood of the closed
cylinder. Therefore,
\[
\|\varphi\|_{H_2(\partial_pQ)}
\leq
\|\widetilde{\Phi}\|_{H_2(J\times U)}
\leq
C\|u^0\|_{C^{2+\alpha}(\overline{\Omega})},
\]
where $C$ depends only on the fixed extension operator and
$\widetilde{\zeta}$ and is uniform for $T'\in[\tau_0,T]$.

Under the same rescaling, the principal matrix and the lower-order term
become $T'a^{ij}$ and $T'B$. Since $T'\in[\tau_0,T]$, the corresponding
ellipticity and structural constants are bounded uniformly above and
below. Lemma~\ref{lem:quoted-gradient} therefore gives, on the fixed
cylinder,
\[
[\nabla\widetilde{u}]_{C_s^{\beta/2}C_x^0}
+
[\nabla\widetilde{u}]_{C_x^\beta C_s^0}
\leq C,
\qquad
\widetilde{u}(s,x)=u(T's,x).
\]
Returning to the original time variable gives
\[
[\nabla u]_{C_t^{\beta/2}C_x^0
([0,T']\times\overline{\Omega})}
\leq
(T')^{-\beta/2}C
\leq
\tau_0^{-\beta/2}C,
\]
while the spatial H\"older seminorm is unchanged. Combining these
estimates with
\[
\|\nabla u\|_{L^\infty([0,T']\times\Omega)}
\leq M_1
\]
yields
\[
\|\nabla u\|_{C^{\beta/2,\beta}
([0,T']\times\overline{\Omega})}
\leq C
\]
with the same exponent and constant for every
$T'\in[\tau_0,T]$.
\end{proof}

\begin{lemma}[Uniform positive-time Schauder bound on finite time intervals]\label{lem:schauder}
Fix $T>0$. Let $0<T'\leq T$, and let $u$ be a
Schauder-classical solution of \eqref{eq:main-pde} on $[0,T']$ satisfying
\[
 \sup_{0\le t\le T'}\|u(t)\|_{C^1(\overline\Omega)}\le M_1.
\]
For every $0<\tau<T'$ there are $\beta\in(0,\alpha]$ and $C<\infty$ such that
\[
 \|u\|_{C^{1+\beta/2,2+\beta}([\tau,T']\times\overline\Omega)}\le C.
\]
For fixed $T$, $\tau$, and $M_1$, the exponent and constant are independent
of $T'\in[\tau,T]$.
\end{lemma}

\begin{proof}
Apply Lemma~\ref{lem:gradient-holder} with $\tau_0=\tau$ and the fixed
time $T$. The lemma gives the same exponent and constant for every
$T'\in[\tau,T]$.

Since $\Omega$ is a bounded connected $C^1$ domain, there exists
$C_\Omega<\infty$ such that every $x,y\in\overline{\Omega}$ can be joined
inside $\overline{\Omega}$ by a piecewise $C^1$ path
$\gamma_{x,y}$ whose length is at most $C_\Omega|x-y|$. Hence, for each
fixed $t$,
\[
|u(t,x)-u(t,y)|
\leq
\int_{\gamma_{x,y}}|\nabla u(t,\xi)|\,d\ell
\leq
C_\Omega
\|\nabla u(t)\|_{L^\infty(\Omega)}
|x-y|.
\]

There also exists $L_\Omega<\infty$ such that every
$x\in\overline{\Omega}$ can be joined to a point of $\partial\Omega$ by
a piecewise $C^1$ path $\gamma_x\subset\overline{\Omega}$ of length at
most $L_\Omega$. Since
\[
u(t,\xi)=u(s,\xi)=0
\qquad
\text{for }\xi\in\partial\Omega,
\]
we obtain
\[
\begin{aligned}
|u(t,x)-u(s,x)|
&=
\left|
\int_{\gamma_x}
\bigl(\nabla u(t,\xi)-\nabla u(s,\xi)\bigr)\cdot d\xi
\right| \\
&\leq
L_\Omega
[\nabla u]_{C_t^{\beta/2}C_x^0}
|t-s|^{\beta/2}.
\end{aligned}
\]
Together with the uniform $L^\infty$ bound for $u$, these estimates give
\[
\|u\|_{C^{\beta/2,\beta}
([0,T']\times\overline{\Omega})}
\leq C.
\]

After decreasing $\beta$, if necessary, the coefficients
$a^{ij}(x,\nabla u)$, the drift term
$(\nabla G)\cdot\nabla u$, and the reaction
$\cT_\eta(x,u)$ have uniformly bounded
$C^{\beta/2,\beta}$ norms.

Choose $\chi\in C^\infty([0,T])$ such that
\[
\chi=0
\quad\text{on }[0,\tau/3],
\qquad
\chi=1
\quad\text{on }[\tau/2,T],
\qquad
|\chi'|\leq \frac{C}{\tau},
\]
and define
\[
v=\chi u
\qquad\text{on }[0,T'].
\]
On $[\tau/4,T']\times\Omega$, the function $v$ satisfies
\[
\begin{cases}
v_t-a^{ij}(x,\nabla u)v_{ij}-(\nabla G)\cdot\nabla v
=
-\mu\chi\cT_\eta(x,u)+\chi'u,
\\[2mm]
v=0
\quad\text{on }[\tau/4,T']\times\partial\Omega,
\\[1mm]
v(\tau/4,\cdot)=0.
\end{cases}
\]
Because $\chi=\chi'=0$ in a neighborhood of $t=\tau/4$, both $v$ and
the right-hand side vanish there. Hence the initial and boundary data are
compatible with the linear equation at
$\{\tau/4\}\times\partial\Omega$.

Set
\[
L_{T'}=T'-\frac{\tau}{4},
\qquad
s=\frac{t-\tau/4}{L_{T'}},
\qquad
\widetilde{v}(s,x)
=
v\left(\frac{\tau}{4}+L_{T'}s,x\right).
\]
For every $T'\in[\tau,T]$,
\[
\frac{3\tau}{4}
\leq
L_{T'}
\leq
T-\frac{\tau}{4}.
\]

The rescaled equation is posed on the fixed cylinder
$[0,1]\times\Omega$. It has zero initial and boundary data, principal
matrix $L_{T'}a^{ij}$, drift $L_{T'}\nabla G$, and right-hand side
multiplied by $L_{T'}$. The preceding bounds for $L_{T'}$, together with
the uniform H\"older bounds for the coefficients and the right-hand side,
give uniform ellipticity and uniformly controlled Schauder data on the
fixed cylinder. The global linear Dirichlet Schauder estimate on the fixed
cylinder therefore gives
\[
\|\widetilde{v}\|_{C^{1+\beta/2,\,2+\beta}
([0,1]\times\overline{\Omega})}
\leq C
\]
with the same constant for every $T'\in[\tau,T]$.

Returning to the original variables and using the upper and lower bounds
for $L_{T'}$ gives
\[
\|v\|_{C^{1+\beta/2,\,2+\beta}
([\tau/4,T']\times\overline{\Omega})}
\leq C.
\]
Since $v=u$ on $[\tau/2,T']$ and
\[
[\tau,T']\subset[\tau/2,T'],
\]
we conclude that
\[
\|u\|_{C^{1+\beta/2,\,2+\beta}
([\tau,T']\times\overline{\Omega})}
\leq C.
\]
\end{proof}

\begin{lemma}[Patching parabolic H\"older estimates]
\label{lem:holder-patching}
Let $0<a<b<c$, and suppose that
\[
w\in
C^{1+\theta_1/2,\,2+\theta_1}
([0,b]\times\overline{\Omega})
\cap
C^{1+\theta_2/2,\,2+\theta_2}
([a,c]\times\overline{\Omega})
\]
for some $\theta_1,\theta_2\in(0,1)$. Then, for every
\[
0<\theta<\min\{\theta_1,\theta_2\},
\]
one has
\[
w\in
C^{1+\theta/2,\,2+\theta}
([0,c]\times\overline{\Omega}).
\]
Moreover, the corresponding norm is bounded by a constant depending only
on the two given H\"older norms, $a$, $b$, $c$, and $\theta$.
\end{lemma}

\begin{proof}
On each of the cylinders
\[
[0,b]\times\overline{\Omega}
\qquad\text{and}\qquad
[a,c]\times\overline{\Omega},
\]
the required estimates follow by lowering the H\"older exponent from
$\theta_1$ or $\theta_2$ to $\theta$. It remains to estimate pairs of
points whose time coordinates lie in different parts of the two
cylinders.

Let $(t,x),(s,y)\in[0,c]\times\overline{\Omega}$, with
$t<a$ and $s>b$. Since
\[
|t-s|\geq b-a,
\]
the corresponding temporal H\"older quotients are bounded by the relevant
supremum norms divided by a fixed positive power of $b-a$.

Now suppose that the two time coordinates are not separated by $b-a$.
Choose $r\in[a,b]$ between $t$ and $s$. For any function $f$ appearing
in the parabolic H\"older norm, the triangle inequality gives
\[
|f(t,x)-f(s,y)|
\leq
|f(t,x)-f(r,x)|
+
|f(r,x)-f(r,y)|
+
|f(r,y)-f(s,y)|.
\]
The first and third terms are controlled by the temporal H\"older
estimates on the two cylinders, while the middle term is controlled by
the spatial H\"older estimate on their overlap.

Apply this argument to $D_xw$ with temporal exponent
$(1+\theta)/2$, and to $D_x^2w$ and $w_t$ with temporal exponent
$\theta/2$. The corresponding spatial estimates follow in the same way
after lowering the spatial exponent to $\theta$. These are precisely the
seminorms defining
\[
C^{1+\theta/2,\,2+\theta}
([0,c]\times\overline{\Omega}),
\]
which proves the result.
\end{proof}

\begin{lemma}[Patching across a single time interface]
\label{lem:interface-patching}
Let $0<t_0<c$, and suppose that
\[
w\in
C^{1+\theta_-/2,\,2+\theta_-}
([0,t_0]\times\overline{\Omega})
\]
and
\[
w\in
C^{1+\theta_+/2,\,2+\theta_+}
([t_0,c]\times\overline{\Omega}),
\]
where $\theta_-,\theta_+\in(0,1)$. Assume that the values of
$w,\ D_xw,\ D_x^2w,\ w_t$
obtained from the two time intervals agree at $t=t_0$. Then, for every
$0<\theta<\min\{\theta_-,\theta_+\}$,
one has
\[
w\in
C^{1+\theta/2,\,2+\theta}
([0,c]\times\overline{\Omega}).
\]
\end{lemma}

\begin{proof}
Since
\[
\theta<\min\{\theta_-,\theta_+\},
\]
the assumed regularity gives
\[
w\in C^{1+\theta/2,\,2+\theta}
([0,t_0]\times\overline{\Omega})
\]
and
\[
w\in C^{1+\theta/2,\,2+\theta}
([t_0,c]\times\overline{\Omega}).
\]
It remains to estimate pairs of points whose time coordinates lie on opposite sides of $t_0$. Let $s<t_0<t$. At a fixed spatial point $x$, the agreement of the corresponding values at $t=t_0$ gives
\[
\begin{aligned}
|D_xw(t,x)-D_xw(s,x)|
&\leq |D_xw(t,x)-D_xw(t_0,x)| \\
&\quad+|D_xw(t_0,x)-D_xw(s,x)| \\
&\leq C\left(|t-t_0|^{(1+\theta)/2}+|t_0-s|^{(1+\theta)/2}\right) \\
&\leq C|t-s|^{(1+\theta)/2}.
\end{aligned}
\]
Similarly,
\[
|D_x^2w(t,x)-D_x^2w(s,x)|\leq C|t-s|^{\theta/2}
\]
and
\[
|w_t(t,x)-w_t(s,x)|\leq C|t-s|^{\theta/2}.
\]
The corresponding estimate for $w$ follows from its temporal H\"older estimates on the two time intervals and the agreement of its values at $t=t_0$.

For general points $(t,x)$ and $(s,y)$ with $s<t_0<t$, the boundedness of $D_x^2w$ gives
\[
\begin{aligned}
|D_xw(t,x)-D_xw(s,y)|
&\leq |D_xw(t,x)-D_xw(s,x)|+|D_xw(s,x)-D_xw(s,y)| \\
&\leq C\left(|t-s|^{(1+\theta)/2}+|x-y|\right).
\end{aligned}
\]
The spatial $C^{1+\theta}$ regularity of $D_xw$ is equivalently controlled by the boundedness of $D_x^2w$ and the spatial $C^\theta$ seminorm of $D_x^2w$, both already available on the two time intervals. The same decomposition gives
\[
|D_x^2w(t,x)-D_x^2w(s,y)|
\leq C\left(|t-s|^{\theta/2}+|x-y|^\theta\right)
\]
and
\[
|w_t(t,x)-w_t(s,y)|
\leq C\left(|t-s|^{\theta/2}+|x-y|^\theta\right).
\]
Together with the corresponding bounds for $w$, these estimates control every seminorm in
\[
C^{1+\theta/2,\,2+\theta}([0,c]\times\overline{\Omega}),
\]
which proves the result.
\end{proof}

\begin{lemma}[Continuation criterion]\label{lem:continuation}
Let $u$ be a locally H\"older-classical solution of \eqref{eq:main-pde} on $[0,T_*)$, where $T_*<\infty$.  If
\[
        \sup_{0\le t<T_*}\norm{u(t)}_{C^1(\overline\Omega)}<\infty,
\]
then $u$ extends as a positive-time classical solution beyond $T_*$.
\end{lemma}

\begin{proof}
Fix $\tau\in(0,T_*)$. For each $T'\in[\tau,T_*)$, apply
Lemma~\ref{lem:schauder} to the solution on $[0,T']$, with $T=T_*$ and
$\tau_0=\tau$ in Lemma~\ref{lem:gradient-holder}. The assumed finite
$C^1$ bound gives
$|u|+|\nabla u|\leq K
\qquad\text{on }[0,T')\times\overline{\Omega}
$
with the same $K$ for every $T'<T_*$. Consequently, there exist
$\beta\in(0,\alpha]$ and $C<\infty$, independent of $T'<T_*$, such that
\[
\|u\|_{C^{1+\beta/2,\,2+\beta}
([\tau,T']\times\overline{\Omega})}
\leq C.
\]
Therefore,
\[
\sup_{T'<T_*}
\|u\|_{C^{1+\beta/2,\,2+\beta}
([\tau,T']\times\overline{\Omega})}
<\infty.
\]

By the definition of the parabolic H\"older norm,
\[
\|u(t)-u(s)\|_{C^2(\overline{\Omega})}
\leq
C|t-s|^{\beta/2},
\qquad
\tau\leq s,t<T_*,
\]
and $u(t)$ is uniformly bounded in
$C^{2+\beta}(\overline{\Omega})$. Choose
$0<\beta'<\beta$.
The estimate
\[
\|f\|_{C^{2+\beta'}(\overline{\Omega})}
\leq
C
\|f\|_{C^2(\overline{\Omega})}^{1-\beta'/\beta}
\|f\|_{C^{2+\beta}(\overline{\Omega})}^{\beta'/\beta}
\]
applied to $f=u(t)-u(s)$ shows that $u(t)$ is Cauchy in
$C^{2+\beta'}(\overline{\Omega})$ as $t,s\uparrow T_*$. Indeed,
\[
\|u(t)-u(s)\|_{C^{2+\beta}(\overline{\Omega})}
\leq
2\sup_{\tau\leq r<T_*}
\|u(r)\|_{C^{2+\beta}(\overline{\Omega})}.
\]
Hence there exists a unique function
$u^{T_*}\in C^{2+\beta'}(\overline{\Omega})$
such that
$u(t)\longrightarrow u^{T_*}
\quad\text{in }C^{2+\beta'}(\overline{\Omega})
\quad\text{as }t\uparrow T_*$.
Since
$u(t,x)=0
\qquad
\text{for }0<t<T_*,\quad x\in\partial\Omega$, the convergence in $C^{2+\beta'}(\overline{\Omega})$ gives
$u^{T_*}=0
\qquad\text{on }\partial\Omega$. Moreover, the positive-time H\"older estimate gives
$u_t\in C([\tau,T']\times\overline{\Omega})$. The time-independent boundary condition may therefore be differentiated
with respect to $t$, yielding
\[
u_t=0
\qquad\text{on }[\tau,T']\times\partial\Omega.
\]
Substitution into the equation gives
\[
\mathcal{K}_{\varepsilon,\nu}[u(t)]
-\mu\cT_\eta(x,u(t))
=0
\qquad
\text{for }x\in\partial\Omega,\quad
\tau\leq t\leq T'.
\]

The nonlinear operator is continuous with respect to the
$C^2(\overline{\Omega})$ norm. Therefore,
\[
\mathcal{K}_{\varepsilon,\nu}[u(t)]
-\mu\cT_\eta(x,u(t))
\longrightarrow
\mathcal{K}_{\varepsilon,\nu}[u^{T_*}]
-\mu\cT_\eta(x,u^{T_*})
\]
in $C^0(\overline{\Omega})$ as $t\uparrow T_*$. Consequently,
$\mathcal{K}_{\varepsilon,\nu}[u^{T_*}]
-\mu\cT_\eta(x,u^{T_*})
=0
\qquad\text{on }\partial\Omega$.
 Thus $u^{T_*}$ satisfies the compatibility condition required by
Lemma~\ref{lem:local}.

The temporal H\"older estimates also extend to $t=T_*$. For example,
for $\tau\leq t<s<T_*$,
\[
\|D_x^2u(t)-D_x^2u(s)\|_{C^0(\overline{\Omega})}
\leq
C|t-s|^{\beta/2}.
\]
Using the convergence of $D_x^2u(s)$ as $s\uparrow T_*$ gives
\[
\|D_x^2u(t)-D_x^2u^{T_*}\|_{C^0(\overline{\Omega})}
\leq
C|T_*-t|^{\beta/2}.
\]
Define
\[
g_*=
\mathcal{K}_{\varepsilon,\nu}[u^{T_*}]
-\mu\cT_\eta(x,u^{T_*}).
\]
The same argument for $D_xu$ and $u_t$, after reducing the exponent if necessary, gives
\[
\|D_xu(t)-D_xu^{T_*}\|_{C^0(\overline{\Omega})}
\leq C|T_*-t|^{(1+\theta_-)/2}
\]
and
\[
\|D_x^2u(t)-D_x^2u^{T_*}\|_{C^0(\overline{\Omega})}
+
\|u_t(t)-g_*\|_{C^0(\overline{\Omega})}
\leq C|T_*-t|^{\theta_-/2}
\]
for every
\[
0<\theta_-<\beta'
\]
and $\tau\leq t<T_*$.
The convergence in
$C^{2+\beta'}(\overline{\Omega})$ also preserves the required spatial
H\"older estimates. Therefore,
\[
u\in
C^{1+\theta_-/2,\,2+\theta_-}
([\tau,T_*]\times\overline{\Omega})
\qquad
\text{for every }0<\theta_-<\beta'.
\]

Apply Lemma~\ref{lem:local} at time $T_*$ with initial value $u^{T_*}$.
There exist $\delta>0$, $\theta_+\in(0,\beta']$, and a unique local
solution $\widehat{u}$ on $[T_*,T_*+\delta]$ satisfying
\[
\widehat{u}(T_*,\cdot)=u^{T_*}
\]
and
\[
\widehat{u}\in
C^{1+\theta_+/2,\,2+\theta_+}
([T_*,T_*+\delta]\times\overline{\Omega}).
\]
The convergence of $u(t)$ to $u^{T_*}$ and the initial condition for $\widehat{u}$ show that their spatial derivatives up to second order agree at $t=T_*$. Since both equations have the same value $g_*$ there,
\[
\lim_{t\uparrow T_*}u_t(t)=g_*=\widehat{u}_t(T_*,\cdot)
\]
in $C^0(\overline{\Omega})$.

Define
\[
\widetilde{u}(t,x)
=
\begin{cases}
u(t,x), & 0\leq t\leq T_*,\\
\widehat{u}(t,x), & T_*\leq t\leq T_*+\delta.
\end{cases}
\]
The values of
\[
\widetilde{u},\qquad
D_x\widetilde{u},\qquad
D_x^2\widetilde{u},\qquad
\widetilde{u}_t
\]
from the two time intervals agree at $t=T_*$. Choose
$0<\theta<\min\{\theta_-,\theta_+\}$. Lemma~\ref{lem:interface-patching} then gives
\[
\widetilde{u}\in
C^{1+\theta/2,\,2+\theta}
([\tau,T_*+\delta]\times\overline{\Omega}).
\]
Uniqueness from Lemma~\ref{lem:stability-local} shows that
$\widetilde{u}$ is the unique continuation of the original solution.
Hence $u$ extends beyond $T_*$.
\end{proof}

\section{Proof of the main theorems}
\label{sec:proof-main}

\begin{proof}[Proof of Proposition~\ref{prop:energy}]
The variation formula \eqref{eq:scalar-variation} follows from the chain differentiation formula and the boundedness of $h_\eta$.  Along the monotone reaction subflow, substitution of $u_t=-\mu\cT_\eta(x,u)$ gives \eqref{eq:energy-dissipation}.  Substitution of the full equation gives the displayed full-flow identity.
\end{proof}

\begin{proof}[Proof of Theorem~\ref{thm:global}]
Lemma~\ref{lem:local} gives a unique Schauder-classical solution on a
nontrivial interval. Let $T_{\max}$ be the maximal existence time.
Proposition~\ref{prop:range} gives
$0\leq u\leq 1$
on every compact subinterval of $[0,T_{\max})$.

Suppose that
$T_{\max}<\infty$. Applying Lemma~\ref{lem:gradient} with
$S=T_{\max}
$ gives
$\sup_{0\leq t<T_{\max}}
\|u(t)\|_{C^1(\overline{\Omega})}
<\infty$. This bound depends only on $T_{\max}$ and the fixed data and is valid on
the entire interval $[0,T_{\max})$. Lemma~\ref{lem:continuation} then
extends the solution beyond $T_{\max}$, which contradicts the definition
of the maximal existence time. Hence
$T_{\max}=\infty$. The finite-time $C^1$ estimate \eqref{eq:finite-time-c1} follows from
Lemma~\ref{lem:gradient} with
$S=T$. To obtain the global Schauder estimate \eqref{eq:positive-holder}, let
$S_{\mathrm{loc}}$ be the initial local-existence time and choose
$0<s<\frac12\min\{T,S_{\mathrm{loc}}\}$. Combine the estimate from Lemma~\ref{lem:local} on
$[0,2s]$ with Lemma~\ref{lem:schauder}, applied with the fixed time $T$, on $[s,T]$. After choosing a common positive exponent below the two H\"older
exponents, Lemma~\ref{lem:holder-patching} gives a $C^{1+\beta_T/2,\,2+\beta_T}$ bound on $[0,T]\times\overline{\Omega}$. Uniqueness follows from Lemma~\ref{lem:stability-local} applied to
solutions with the same initial data.
\end{proof}

\begin{proof}[Proof of Theorem~\ref{thm:stability}]
Apply Lemma~\ref{lem:stability-local} on $[0,T]$ for arbitrary $T>0$.  Since both solutions are global, $T$ is arbitrary, and \eqref{eq:stability} follows for every $t\ge0$.
\end{proof}

\section{Application to 3D and 3D+time nuclei data}
\label{sec:interpretation}

The model \eqref{eq:main-pde} is intended for the segmentation of touching
and dividing cell nuclei, with fixed parameters $\nupar>0$ and $\eps>0$.
The interpretation of the evolving function, the reaction term, and the
decision level $q$ is given in Appendix~\ref{app:model-interpretation}.

For a single 3D microscopy image, $d=3$ and
\[
x=(x_1,x_2,x_3).
\]
The function $u(t,x)$ evolves in artificial segmentation time. Image-edge
information enters through the prescribed coefficient $G$, while information
about nearby candidate nuclei enters through the interaction weight
$\Lambda$. The latter is constructed from candidate centers,
graph-neighborhood relations, prescribed neighboring-candidate maps, and
smooth spatial weights.

For 3D+time data, $d=4$ and
\[
x=(x_1,x_2,x_3,\theta),
\]
where $\theta$ denotes physical image time. The intensity function
\[
I^0(x_1,x_2,x_3,\theta)
\]
and the coefficients $G$ and $\Lambda$ are defined on the four-dimensional
computational domain. They may be constructed frame by frame in the spatial
variables or locally in space and physical time using a prescribed scaled
space-time metric. The unknown
\[
u(t,x_1,x_2,x_3,\theta)
\]
then evolves in artificial segmentation time $t$ according to the same
equation.

The global existence and stability results apply in both cases once the
prescribed coefficient functions satisfy the assumptions of the analysis.

\section{Conclusion}

We introduced a regularized subjective-surface model for touching and dividing cell nuclei and established a unique global classical solution whose restriction to every finite time interval is Schauder-classical for fixed $\nupar>0$ and $\eps>0$. The analysis proves preservation of the physical range, regularity on every finite time interval, and fixed-coefficient $L^\infty$ nonexpansiveness. The central a priori step is the global spatial-gradient bound obtained by combining gradient estimates near $\partial\Omega$ with interior gradient estimates. Positive-time H\"older-gradient theory and Schauder estimates then yield global continuation.

The coefficient $G$ incorporates image-edge information, while the graph-based construction of $\Lambda$ incorporates fixed information from neighboring nucleus candidates into the monotone reaction term. The resulting existence and stability theory gives the mathematical foundation for the proposed model of touching and dividing nuclei in 3D and 3D+time microscopy data.

\appendix
\section{Verification of the global H\"older-gradient estimate}
\label{app:gradient-input}

Fix $0<\tau_0\le T$. For each $T'\in[\tau_0,T]$, apply
\cite[Theorem~4.7]{Lieberman1986} to
$Q_{T'}=(0,T')\times\Omega$. In the notation of that theorem, set
\[
A^{ij}(X,z,p)=a^{ij}(x,p),\qquad
F(X,z,p)=\nabla G(x)\cdot p-\mu\Lambda(x)H_\eta(z-q),
\]
and choose the admissible exponent $\gamma=2$. Proposition~\ref{prop:range} and the finite-time $C^1$ estimate give
\[
|z|\le1,\qquad |p|\le M_1.
\]
On this range, Lemma~\ref{lem:ellipticity} gives
$\nupar I\le(A^{ij})\le(\nupar+g^*)I$. Moreover, $A^{ij}$ is independent of $t$ and $z$, and its required
$x$- and $p$-derivatives are bounded because
$G\in C^{3+\alpha}(\overline\Omega)$ and $\eps>0$. Thus
\[
A_z^{ij}=A_t^{ij}=0,\qquad
\sup_{|p|\le M_1}
\bigl(|D_xA|+|D_pA|+|D_{xp}^2A|+|D_{pp}^2A|\bigr)<\infty.
\]
The lower-order term satisfies
\[
|F(X,z,p)|\le C(1+|p|),\qquad
F_p=\nabla G,\qquad F_{pp}=F_{zp}=0,
\]
while $F_z$, $F_{zz}$, and $D_xF$ are controlled by the fixed norms of
$G$, $\Lambda$, and $H_\eta$.

The hypotheses in \cite[(4.17), p.~378]{Lieberman1986} hold with
\[
K=1+M_1,\qquad \lambda_K=\nupar.
\]
Condition~(4.25) in \cite[p.~381]{Lieberman1986} becomes, since
$\gamma=2$,
\[
|\partial_zA^{ij}|+\sum_{k=1}^d|\partial_{p_k}A^{ij}|
+|A_x^{ij}|+|A_t^{ij}|+|F|\le\mu_K,
\]
and follows from
\[
A_z^{ij}=A_t^{ij}=0,\qquad
\sum_{i,j,k}|\partial_{p_k}A^{ij}|
+\sum_{i,j,k}|\partial_{x_k}A^{ij}|+|F|\le C(M_1).
\]

The prescribed data are
\[
\varphi(t,x)=
\begin{cases}
u^0(x),&t=0,\\
0,&x\in\partial\Omega.
\end{cases}
\]
Set $s=t/T'$ and work on $(0,1)\times\Omega$. Choose a bounded extension
operator
\[
E:C^{2+\alpha}(\overline\Omega)\longrightarrow C^{2+\alpha}(U),
\]
where $U\supset\overline\Omega$, and choose
$\widetilde\zeta\in C^\infty(J)$ on an interval $J\supset[0,1]$ with
$\widetilde\zeta(0)=1$. Define
\[
\widetilde\Phi(s,x)=\widetilde\zeta(s)\,Eu^0(x).
\]
Since $u^0=0$ on $\partial\Omega$, this function agrees with the rescaled
prescribed data on the parabolic boundary. By
\cite[p.~350]{Lieberman1986},
\[
\|\varphi\|_{H_2(\partial_pQ)}
\le\|\widetilde\Phi\|_{H_2(J\times U)}
\le C\|u^0\|_{C^{2+\alpha}(\overline\Omega)},
\]
with a constant uniform for $T'\in[\tau_0,T]$.

After rescaling, the principal matrix and lower-order term are $T'A^{ij}$
and $T'F$. Since $T'\in[\tau_0,T]$, the ellipticity constants, coefficient
bounds, and prescribed-data norms are uniform. Hence
\[
\|\nabla\widetilde u\|_{C^{\beta/2,\beta}
([0,1]\times\overline\Omega)}\le C,
\qquad
\widetilde u(s,x)=u(T's,x),
\]
for some $\beta\in(0,1)$. Returning to $t$ gives
\[
[\nabla u]_{C_t^{\beta/2}C_x^0}
\le (T')^{-\beta/2}C
\le\tau_0^{-\beta/2}C,
\]
while the spatial seminorm is unchanged. Therefore,
\[
\|\nabla u\|_{C^{\beta/2,\beta}
([0,T']\times\overline\Omega)}\le C
\]
with the same exponent and constant for every $T'\in[\tau_0,T]$.

Lemma~\ref{lem:schauder} uses this estimate to obtain the corresponding
positive-time $C^{1+\beta/2,2+\beta}$ bound.

\section{Model interpretation and threshold selection}
\label{app:model-interpretation}

This appendix gives the modeling interpretation of the evolution equation and clarifies the role of the decision level $q$.  The discussion is intended to complement the analytical results.

\subsection{Interpretation of the evolving function}

The model evolves a single function $u=u(t,x)$ whose selected level sets represent nuclear boundaries.  At each artificial segmentation time $t$, the value $u(t,x)$ indicates whether a point $x$ lies in the nuclear region, in the background, or in the transition zone near a nuclear boundary.  The evolution equation is
\begin{equation}\label{eq:appendix-model}
\partial_t u
=
\nupar\Delta u
+
\left(\eps^2+|\nabla u|^2\right)^{1/2}
\operatorname{div}\!\left(
G(x)\frac{\nabla u}{\left(\eps^2+|\nabla u|^2\right)^{1/2}}
\right)
-
\mu\Lambda(x)H_\eta(u-q).
\end{equation}
It combines three effects.

\paragraph{Linear smoothing.}
The term
\[
\nupar\Delta u
\]
smooths the evolving function and suppresses small-scale oscillations.  Because $\nupar>0$, it also contributes the uniform ellipticity used in the analysis.

\paragraph{Edge-sensitive geometric motion.}
The term
\[
\left(\eps^2+|\nabla u|^2\right)^{1/2}
\operatorname{div}\!\left(
G(x)\frac{\nabla u}{\left(\eps^2+|\nabla u|^2\right)^{1/2}}
\right)
\]
is the regularized subjective-surface contribution.  The prescribed coefficient $G$ is constructed from image information.  It is typically smaller near strong image edges and larger in relatively homogeneous regions.  Consequently, the evolving surface moves more freely in smooth regions and slows near likely nuclear boundaries.  The parameter $\eps>0$ regularizes the denominator and prevents degeneracy when $\nabla u=0$.

\paragraph{Neighbor-interaction term.}
The term
\[
-\mu\Lambda(x)H_\eta(u-q)
\]
incorporates fixed information about nearby candidate nuclei.  Here $\Lambda(x)\ge0$ identifies regions where neighboring-candidate information is important, $\mu>0$ controls the strength of the interaction, $q$ is the decision level, and $H_\eta$ is a smooth approximation of a step function.  As $u$ crosses the level $q$, the factor $H_\eta(u-q)$ activates the interaction term.  The evolution can therefore respond differently in regions where two nuclei are believed to touch or where one nucleus is dividing.

\subsection{Choice and interpretation of the decision level}

The parameter $q\in(0,1)$ is the decision level used both to interpret the segmentation and to activate the monotone reaction.  Since the reaction profile appears as
\[
H_\eta(u-q),
\]
we have
\[
u\le q-\eta
\quad\Longrightarrow\quad
H_\eta(u-q)=0,
\]
so the neighboring-nucleus penalty is inactive, whereas
\[
 u\ge q+\eta
\quad\Longrightarrow\quad
H_\eta(u-q)=1,
\]
so the penalty is fully active.  For values in the transition band $q-\eta<u<q+\eta$, the reaction turns on smoothly.

The analytical assumptions require
\[
q\in(0,1),
\qquad
0<\eta<\min\{q,1-q\}.
\]
This condition ensures that the entire transition interval
\[
[q-\eta,q+\eta]
\]
lies inside the physical range $[0,1]$.

From the application viewpoint, $q$ is the level used to distinguish the nuclear region from the background.  When $u\approx1$ represents nucleus and $u\approx0$ represents background, the symmetric choice
\[
q=\frac12
\]
is a natural default.  In that case,
\begin{equation}\label{eq:appendix-segmented-region}
\Omega_t^{\mathrm{nuc}}
=
\left\{x\in\Omega:u(t,x)>\frac12\right\}.
\end{equation}
The existence and stability results do not single out an optimal value of $q$; the theorem treats $q$ as a prescribed parameter satisfying the above constraints.  In numerical work, $q$ may instead be selected by validation against annotated data, by an image-dependent thresholding criterion, or by a sensitivity study.

\subsection{Relation to the analytical results}

The global theorem shows that, for fixed admissible parameters and compatible initial data, the evolution in \eqref{eq:appendix-model} exists for all artificial times, is unique, remains in the physical range $0\le u\le1$, satisfies Schauder estimates on every finite time interval, and depends nonexpansively on the initial data in $L^\infty$.  These conclusions justify the use of the level-set interpretation in \eqref{eq:appendix-segmented-region} throughout the evolution.

\section*{Data Availability}

No datasets were generated or analysed during the current study.

\end{document}